\documentclass[reqno]{amsart}

\usepackage{amssymb, amsthm, amsmath, amsfonts}
\usepackage{tikz,cancel, enumerate, url, mathrsfs}

\usepackage[hypertexnames=false]{hyperref}  
\numberwithin{equation}{section}
\usepackage{color}

\usepackage{graphicx}

\newtheorem{definition}{Definition}[section]
\newtheorem*{definition*}{Definition}
\newtheorem{assumption}{Assumption}[section]
\newtheorem{proposition}{Proposition}[section]
\newtheorem*{proposition*}{Proposition}
\newtheorem{lemma}{Lemma}[section]
\newtheorem*{lemma*}{Lemma}
\newtheorem{theorem}{Theorem}[section]
\newtheorem*{theorem*}{Theorem}
\newtheorem{coro}{Corollary}[theorem]
\newtheorem*{coro*}{Corollary}

\newtheorem*{example*}{Example}

\newtheorem*{exercise*}{Exercise}
\newtheorem{remark}{Remark}[section]
\newtheorem*{remark*}{Remark}

\newcommand{\bfH}{\mathbf{H}}

\newcommand{\bfv}{\mathbf{v}}

\newcommand{\bfx}{\mathbf{x}}
\newcommand{\bfy}{\mathbf{y}}
\newcommand{\bfz}{\mathbf{z}}

\newcommand{\bfnu}{{ \boldsymbol{\nu} }}

\newcommand{\bbR}{\mathbb{R}}

\newcommand{\cB}{\mathcal{B}}

\newcommand{\cD}{\mathcal{D}}

\newcommand{\cG}{\boldsymbol{\mathcal{G}}}

\newcommand{\cL}{\mathcal{L}}

\newcommand{\cN}{\mathcal{N}}

\newcommand{\scrH}{\mathscr{H}}

\newcommand{\bfal}{{\boldsymbol{\alpha}}}

\newcommand{\al}{\alpha}

\newcommand{\del}{\delta}
\newcommand{\Del}{\Delta}

\newcommand{\Gam}{\Gamma}

\newcommand{\eps}{\varepsilon}

\newcommand{\om}{\omega}
\newcommand{\Om}{\Omega}

\newcommand{\kap}{\kappa}
\newcommand{\sig}{\sigma}

\newcommand{\txtin}{\quad\text{in}\quad}

\newcommand{\txtor}{\quad\text{or}\quad}
\newcommand{\txtfor}{\quad\text{for}\quad}
\newcommand{\txtand}{\quad\text{and}\quad}

\newcommand{\txtforall}{\quad\text{for all}\quad}

\newcommand{\grad}{\nabla}

\newcommand{\Rn}{\mathbb{R}^n}

\newcommand{\cl}{\overline}
\newcommand{\half}{\frac{1}{2}}

\newcommand{\Dn}[1]{\frac{\partial #1}{\partial \nu}}

\newcommand{\tl}{\tilde}

\newcommand{\p}{\partial}

\newcommand{\dfn}{\mathrel{\mathop:}=}

\newcommand{\intOm}{\int_{\Omega}}

\newcommand{\ltom}{{L^2(\Omega)}}

\newcommand{\fmbox}[1]{\fbox{$\displaystyle {#1}$}}

\newcommand{\ext}{\text{ext}}
\newcommand{\inter}{\operatorname{int}}

\newcommand{\dist}{\operatorname{dist}}
\newcommand{\supp}{\operatorname{supp}}

\newcommand{\lp}{\big(\!\!\!\;\big(}
\newcommand{\rp}{\big)\!\!\!\;\big)}

\renewcommand{\lg}{\langle}

\newcommand{\rg}{\rangle}

\newcommand{\ds}{\displaystyle}

\title[Nonlocal Biharmonic]{A nonlocal biharmonic operator and its connection with the classical analogue
}

\author[P.~Radu]{Petronela Radu}
\email{pradu@unl.edu}

\author[D.~Toundykov]{Daniel Toundykov}
\email{dtoundykov@unl.edu}

\author[J.~Trageser]{Jeremy Trageser}
\email{s-trages1@math.unl.edu}

\address{University of Nebraska-Lincoln, Department of Mathematics, Lincoln, NE 68588}



\thanks{The research of the first author was partially supported by the National Science Foundation under Grant DMS-0908435.\\
The research of second author was partially supported by the National Science Foundation under Grant DMS-1211232.}

\def\Ln{\cL_{\bfal_{\del_n}}}

\begin{document}

\begin{abstract}

We consider a nonlocal operator as a natural generalization to the biharmonic operator that arises in thin plate theory. The operator is built in the nonlocal calculus framework defined in \cite{DGLZ2013} and connects with the recent theory of peridynamics. This framework enables us to consider non-smooth approximations to fourth-order elliptic boundary value problems. For these systems we introduce nonlocal formulations of the clamped and hinged boundary conditions that are well-defined even for irregular domains.  We demonstrate existence and uniqueness of solutions to these nonlocal problems and demonstrate their $L^2$-strong convergence to functions in $W^{1,2}$ as the nonlocal interaction horizon goes to zero. For regular domains we identify these limits as the weak solutions of the corresponding classical elliptic boundary value problems.  As a part of our proof we also establish that the nonlocal Laplacian of a smooth function is Lipschitz continuous.
\end{abstract}

\maketitle

\section{Introduction}
Classical models of continuum mechanics give rise to fourth-order elliptic PDEs  describing transversal deformations of thin plates, shells or beams, possibly in coupling with additional equations quantifying shear forces \cite{Sel,Yao}. Under some regularity conditions on the boundary, solutions to a fourth-order elliptic boundary value problem generally acquire four orders of weak differentiability with respect to the regularity of the interior forcing term. In two dimensions, systems with square integrable forcing typically possess weak solutions  that have at least $W^{2,2}\subset W^{1,\infty}$ Sobolev regularity. In particular, such solutions (in the 2D case) are necessarily continuous which makes it non-trivial to account for irregularities, such as cracks.  On the other hand, prime examples of plate structures are suspension bridges, where the dynamic formation of cracks and their evolution is of great interest; however, discontinuities corresponding to damage preclude the inclusion of smoothness assumptions on the solutions. 

A proposed paradigm for investigating less regular solutions is to replace the classical operators of elasticity theory with suitable approximations that replace derivatives with singular integral operators. This approach is prompted by physical considerations, such as describing the stress at a point via the cumulative interaction with points from its neighborhood; this interaction is often captured through a suitable integral kernel.

Nonlocal versions of the classical Laplace operator have been investigated in various settings and phenomena: nonlocal diffusion \cite{Rossi}, population and swarm models \cite{CF, MK}, and image processing \cite{Gilboa2009}. Recently, this nonlocal operator was used in the peridynamic theory developed by Silling \cite{Silling2000} to describe the evolution of damage in solids.  The relaxed regularity conditions permit fractures
to be represented in the solution of the system prior to discretization rather than being considered through separate ad-hoc frameworks. Within the context of the state-based formulation of peridynamics, nonlocal models of beams and plates were developed in \cite{GF1,GF2}; however, the nonlocality in these models resembles more the nonlocal Laplacian structure (see \cite{DGLZ2013}). 
Nonlocal versions of higher-order operators have also been considered in phase separation problems and Cahn-Hilliard equations, see  \cite{Rossi, Han, NO}.

By replacing differential operators with integral operators it is possible to have well-defined solutions with discontinuous (in fact, with lack of any Sobolev regularity) behavior. Motivated by the aforementioned developments we consider a nonlocal version of the biharmonic operator, obtained by iterating the nonlocal Laplacian. We show that solutions of the nonlocal biharmonic equation with nonlocal analogues of hinged or clamped boundary conditions require minimal integrability assumptions; moreover, in the limit they recover the classical weak solutions of the corresponding elliptic fourth-order problem. 

As previous works have also demonstrated, the nonlocal setting offers an alternative way to study problems in a weak formulation. In contrast with the classical framework which often considers \emph{regularized solutions} to a problem and passes to the limit in a weaker topology to obtain a \emph{less regular} solution, in the nonlocal framework  the investigation starts in a weaker topology, and then as the support of the kernel shrinks, the corresponding weak solutions converge (in the weak topology) to a \emph{more regular} solution. The diagram below illustrates the two approaches in the elliptic framework:
\begin{center}
\begin{tabular}{c|c}
\bf Modeling approach & \bf Convergence\\[0.05in]
\hline\\
Local/Classical & regularized approximations $\xrightarrow{W^{k,2}}$ weak ($W^{k,2}$) solution\\[.05in]
Nonlocal & $L^2$ approximations $\xrightarrow{L^2}$ weak ($W^{k,2}$) solution
\end{tabular}
\end{center}
Finally, the results presented here do not rely on the scalar setting, so they are transferrable to the vectorial framework as introduced in \cite{DGLZ2013}. 

\medskip


\noindent {\bf Boundary conditions.} Although a nonlocal form of the biharmonic operator appears natural, the form of the boundary conditions (BC) is more delicate. Two types of homogeneous BC are fundamental in plate systems: hinged ($u=0, \,\Delta u =0$ on the boundary of the domain) and clamped ($u=0, \frac{\partial u}{\partial \nu}=0$). Since nonlocal operators are associated with collar-type constraints---imposed on sets of positive Lebesgue measure that surround the domain---we need to find appropriate nonlocal generalizations that converge to the classical conditions in the limit as the approximations improve. To our knowledge, this is the first work that deals with integral approximations of higher order elliptic operators coupled with first and second-order boundary conditions. One of the features of the nonlocal approach
is that approximations and their boundary conditions can be formulated for very rough domains. The details concerning the boundary regularity required in our results are summarized in the following table: 
\begin{center}
\begin{tabular*}{\textwidth}{l | l}
\bf Regularity of  & \bf Conclusion for the nonlocal biharmonic \\ 
\bf (bounded open) domain & \bf problem\\[.01in]
\hline\\[.001in]
Lipschitz & Well-posedness of the nonlocal problem in $L^2(\Om)$.\\[.1in]
Class $C^1$
 & Convergence in $L^2(\Om)$ to distributional solution.\\
 (relaxation may be possible) \\[.1in] 
Class $C^2$ or convex $C^1$ & \parbox{3in}{Convergence to \emph{weak} solution of elliptic hinged or clamped problem.}\\[.2in]
Class $C^4$  & \parbox{3in}{Convergence to \emph{regular} solution of elliptic hinged or clamped problem.}
\end{tabular*}
\end{center}

\medskip


Some of the main tools used in local theories are compactness results (such Gagliardo-Nirenberg-Sobolev embedding theorems), estimates obtained through\break Poincar\'e-type inequalities, and in parabolic/elliptic  theory  one often uses the gain in the smoothness of solutions. For nonlocal problems in the case of operators with weakly-singular kernels, however, these methods do not apply. For example, solutions to the Poisson problem gain two degrees of regularity over the forcing term; however, in the nonlocal scenario, there is no improvement in regularity. A nonlocal version of Poincar\'e's inequality exists; however, it does not yield higher $L^p$ integrability for the solution from bounds of the nonlocal gradient. In addition, Sobolev embedding-type theorems do not apply, unless the kernel exhibits a strong (i.e. non-integrable) singularity. One of the key tools utilized in the nonlocal setting is a compactness result due to Bourgain, Brezis, and Mironescu \cite{Bourgain} when the kernel is ``almost-integrable". Such compactness theorems have been used and further developed by Du and Mengesha \cite{DuMengesha} (which is also an inspiration for our work), as well as by Ignat and his collaborators \cite{Ignat}. 

The contributions of this paper to the development of nonlocal theories and understanding the connections between local and nonlocal models are as follows:
\begin{itemize}

\item we study a higher order nonlocal analogue to the biharmonic operator for which we formulate nonlocal equivalents of clamped and hinged boundary conditions; we establish well-posedness of solutions to these nonlocal boundary values problems supplemented with nonlocal BCs;

\item we show $L^2$ strong convergence of the sequence of nonlocal solutions to the classical one as the radius of the support of the kernel in the integral operator goes to zero;

\item under certain assumptions on the kernel of the integral operator (which include singular and weakly singular cases) we prove several properties for the nonlocal operators such as: $L^\infty$ bound on the biharmonic, as well as Lipschitz continuity of the nonlocal Laplacian applied to a sufficiently regular function (see Proposition \ref{lapexists}, respectively, Theorem \ref{lapholder}).  
\end{itemize}

\medskip




This paper is organized as follows. The nonlocal operator framework is outlined in Section 2. 
 Section 3 addresses the connection between the local and nonlocal operators, by proving convergence (and rates of convergence results) as the size of the support of the kernel shrinks to zero. Section 4 presents the well-posedness proof for nonlocal steady state problem with hinged and clamped boundary conditions. 

\section{Background}
This section contains definitions of several fundamental integral operators and  associated function spaces that will be central to our work. Henceforth,  $\Omega$ will denote an open bounded connected subset of $\mathbb{R}^d$;  in some results, we will specialize to $d=2$ as the more interesting case due to the integrability conditions and  because of its connection to the thin plate theory. As we will see in the sequel the existence results for solutions to the nonlocal problems defined below do not require any regularity conditions on $\Om$.  However, in order to establish connections between nonlocal and classical systems, we will need to impose some smoothness conditions on the boundary. The open subdomain  $\Omega'$ will be compactly contained in $\Omega$ (see Figure \ref{domain} for an illustration of the domain $\Omega'$ with its boundary $\Omega \setminus \Omega'$). In addition, for  higher-order conditions we will also consider a subdomain $\Om''$ compactly contained in $\Om'$.

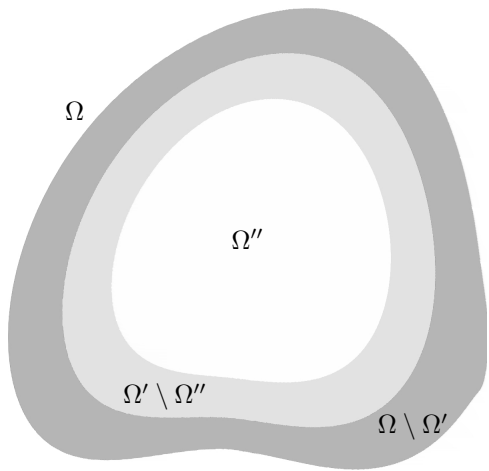
\begin{figure}
 \centering
   \begin{tikzpicture}
     \node at (0,0) {\includegraphics[width=0.5\linewidth]{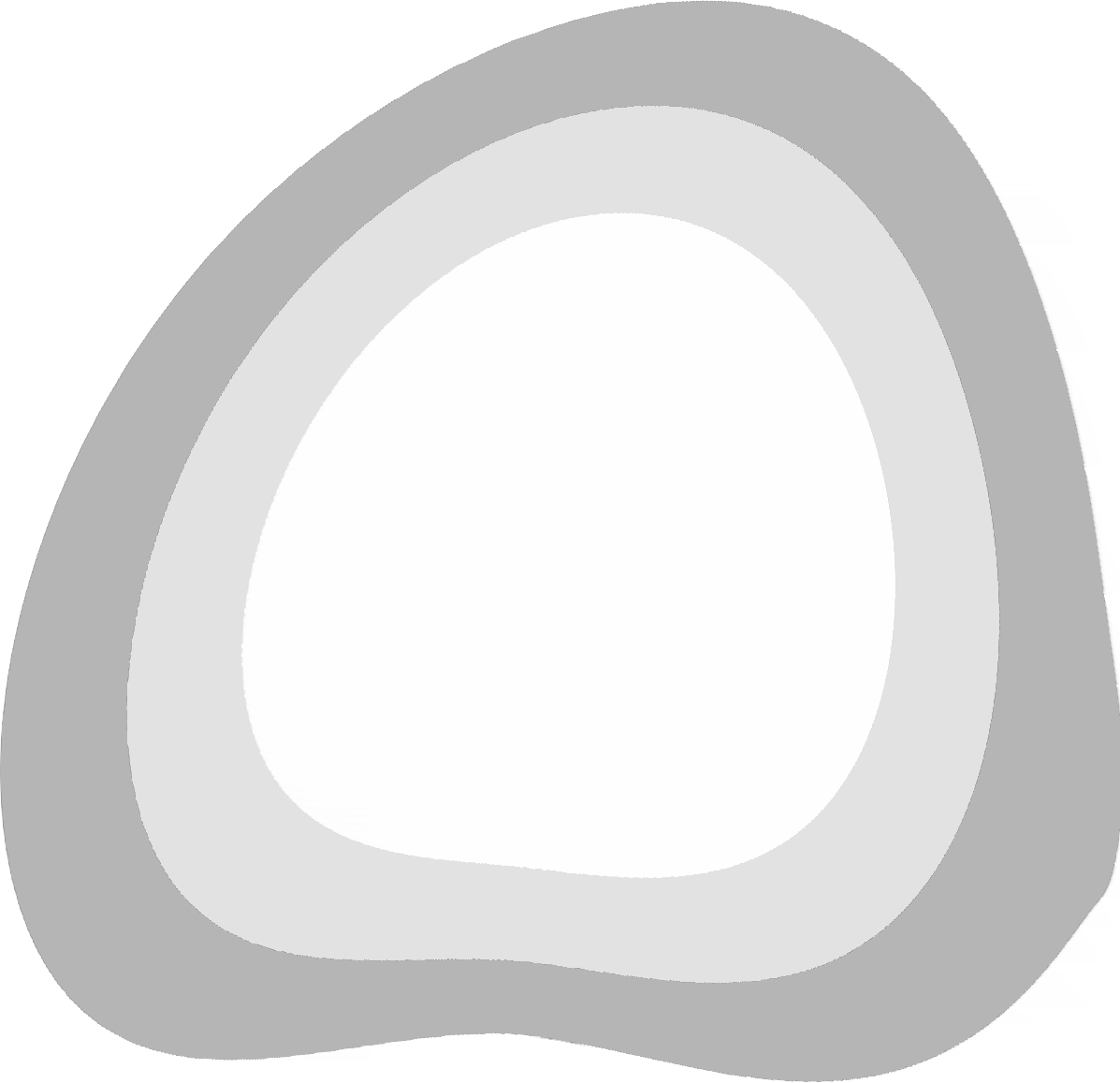}};
     \node at (0,0) {$\Omega''$};
     \node at (-1.1,-2.1) {$\Omega' \setminus \Omega''$};
     \node at (-2.3,1.7) {$\Omega$};
     \node at (2.2,-2.5) {$\Omega \setminus \Omega'$};
  \end{tikzpicture}
\caption{The total domain $\Omega$ with subdomains $\Om''\subset \Om'\subset \Om$. The  ``collar layers" are $\Omega \setminus \Omega'$ and $\Om'\setminus \Om''$.}
\label{domain}
\end{figure}

\subsection{Operators}
As in \cite{DGLZ2013}, we introduce several nonlocal peridynamic operators. 
\begin{definition}[Nonlocal divergence] For a function $\bfnu : \Omega \times \Omega \to \mathbb{R}^k$ and an antisymmetric vector-valued kernel $\bfal: \Omega \times \Omega \to \mathbb{R}^k$, the nonlocal divergence operator $\mathcal{D}_{\bfal}$ is a function-valued  mapping whose image $ \cD_\bfal[\bfnu]$ is defined by $\cD_\bfal[\bfnu] :\Omega \to \bbR$ 
\begin{equation*}
\mathcal{D}_{\bfal}[\bfnu](\bfx)\dfn  \int_{\Omega} (\bfnu(\bfx,\bfy)+\bfnu(\bfy,\bfx)) \cdot \bfal (\bfx,\bfy) d\bfy \quad \text{for}\quad \bfx\in \Om\,.
\end{equation*}
\end{definition}

\begin{definition}[Nonlocal two-point gradient] \label{def:nonlocal-grad} Given a function $u(\bfx): \mathbb{R}^d \to \mathbb{R}$, the formal adjoint of $\mathcal{D}_{\bfal}$ is the nonlocal two-point gradient operator $\cG_{\bfal}: u \mapsto \cG_{\bfal}$ where $\cG_\bfal: \Omega\times \Om \to \Omega$ is given by 
\begin{equation*}
\cG_{\bfal}[u](\bfx,\bfy) = (u(\bfy)-u(\bfx)) \bfal(\bfx,\bfy)
\quad\text{for}\quad(\bfx,\bfy) \in \Omega \times \Omega\,.
\end{equation*}
\end{definition}

\begin{definition}[Nonlocal normal derivative]  For a function $\bfnu : \Omega \times \Omega \to\mathbb{R}^k$ and an antisymmetric vector-valued function $\bfal: \Omega \times \Omega \to\mathbb{R}^k$, the nonlocal normal operator is a mapping $\mathcal{N}_{\bfal}: \bfnu \mapsto  \cN_\al[\bfnu]$ where $\mathcal{N}_\al[\bfnu] :\Om \to \bbR$ is given   by
\begin{equation}\label{def:nonlocal-normal}
\mathcal{N}_{\bfal}[\bfnu](\bfx):=  \int_{\Omega \setminus \Omega'} (\bfnu(\bfx,\bfy)+\bfnu(\bfy,\bfx)) \cdot \bfal (\bfx,\bfy) d\bfy \txtfor \bfx \in \operatorname{int}(\Omega').
\end{equation}
\end{definition}


\begin{definition}[Nonlocal Laplacian]\label{def:laplace}
Let $u: \Omega \to \mathbb{R}$ and $\mu = \bfal^2$ where $\bfal: \Omega \times \Omega \to \mathbb{R}^k$ is an antisymmetric vector-valued function. The nonlocal Laplace operator is defined by:
\begin{equation*}
\begin{split}
\cL_{\bfal}[u](\bfx) := \mathcal{D}_{\bfal} \left[ \cG_{\bfal}[u] \right] = 2 \int_{\Omega} (u(\bfy) - u(\bfx)) \mu(\bfx,\bfy) d\bfy 
\quad\text{for}\quad\bfx \in \Omega\,. 
\end{split}
\end{equation*}
\end{definition}
\noindent It was shown in \cite[Prop. 5.4]{DGLZ2013},  that if $\bfal^2$ is 
formally replaced by distributional application of 
$\frac{1}{2} \Delta_\bfy \delta(\bfy-\bfx)$, then $\cL_{\bfal}$ can be identified, in the sense of distributions, with the Laplace operator $\Delta_\bfx$. 

Following the above framework we define the nonlocal biharmonic operator:
\begin{definition}[Nonlocal biharmonic] Let $\bfal: \Omega \times \Omega \to \mathbb{R}^k$ be an antisymmetric vector valued function and $u: \Omega \to \mathbb{R}$. We define the nonlocal biharmonic by 
\begin{equation}
\mathcal{B}_{\bfal}[u](\bfx) = \cL_{\bfal}[\cL_{\bfal}[u]], \txtfor\bfx \in \Omega. 
\end{equation}
\end{definition}

\subsection{Continuity and integrability} 
Let us recall and prove several results regarding the integrability and continuity for some of the above nonlocal operators. 

First let us recall a nonlocal version of the ``integration by parts" theorem, a simple consequence of the fact that the integrand is antisymmetric:
\begin{proposition}[Nonlocal integration by parts \cite{DGLZ2013}]\label{parts}  Let $\Omega \subset \mathbb{R}^d$ be open, $u, w : \Omega \to \mathbb{R}$, $\bfal: \Omega \times \Omega \to \mathbb{R}^k$ be antisymmetric, and $\bfnu: \Omega \times \Omega \to \mathbb{R}^k$. When $\cD_{\bfal} v \in \ltom $ and $\cG_{\bfal} \in L^2(\Om\times \Om)$ we have

\begin{equation}\label{nonlocal-by-parts}
\int_{\Omega} u(\bfx) \mathcal{D}_{\bfal} [\bfnu] d\bfx = - \int_{\Omega} \int_{\Omega}   \cG_{\bfal} [u] \cdot \bfnu d\bfy d\bfx\,.
\end{equation}
As a special case, when $\mathcal{L}_{\bfal}[u], w \in L^2(\Omega)$ and $\cG_\bfal[u],\cG_{\bfal}[w] \in L^2(\Omega\times \Om)$, we have
\begin{equation}\label{nonlocal-by-parts-lap}
\int_\Omega \mathcal{L}_\bfal[u] w d \bfx = - \int_\Omega \int_\Omega \cG_\bfal[u]\cdot\cG_\bfal[w] d \bfy d \bfx.
\end{equation}
As remarked subsequently, the identity \eqref{nonlocal-by-parts-lap} by definition applies to the weak nonlocal Laplacian introduced below in Definition \ref{def:nonlocalweaklap}.
\end{proposition}

\begin{remark} 
The terms in ``nonlocal" integration by parts already incorporate the information from the collar of the domain, thus boundary terms do not explicitly appear in \eqref{nonlocal-by-parts} or \eqref{nonlocal-by-parts-lap}.
\end{remark}

The next result provides us with an inequality that gives an upper bound for the nonlocal gradient in terms of its classical counterpart:

\begin{theorem}[{c.f. \cite[Thm. 1]{Bourgain}}]\label{bourgain2} 
Let $\Om$ be a bounded Lipschitz domain. 
Suppose $f \in W^{1,p}(\Omega), 1 \leq p < \infty$ and let $\xi \in L^1(\mathbb{R}^d), \xi \geq 0$. Then
\begin{equation*}
\intOm \intOm \frac{|f(\bfx)-f(\bfy)|^p}{|\bfx-\bfy|^p} \xi(\bfx-\bfy) d \bfx d \bfy \leq C \|f \|_{W^{1,p}(\Omega)}^p \| \xi \|_{L^1(\Omega)},
\end{equation*}
where $C$ depends only on $p$ and $\Omega$.
\end{theorem}
\begin{remark}
The result \cite[Thm. 1]{Bourgain} focuses on smooth domains, but the proof only requires the function in question to have a  $W^{1,p}(\Rn)$ extension. Thus it suffices for $\Om$ to satisfy a strong locally Lipschitz condition, or equivalently be bounded and Lipschitz  (see, for example, \cite[p. 83 and Thm. 5.24 on p. 154]{adams}).
\end{remark}


%

We will need the following assumption, first introduced in \cite{Bourgain}, on the family of kernels used in our nonlocal formulations.
\begin{assumption}[Kernel $\bfal$]\label{as:alpha}
For $\del>0$ let $\rho_\delta$ be a radial compactly-supported mollifier satisfying
\begin{equation}\label{mollifier}
\rho_\del: C^\infty(\bbR^+;\bbR^+),\quad
\int_{\bbR^d} \rho_\del(|\bfx|) d\bfx =1, \quad \supp(\rho_\del)\subset [0,\del)\,.
\end{equation}
Define
\begin{equation}\label{def:alpha}
\bfal_\delta(\bfx,\bfy) := \frac{\sqrt{\rho_\delta(|\bfx-\bfy|)}}{|\bfx-\bfy|^{2}} (\bfx-\bfy)\,.
\end{equation}
\end{assumption}
Henceforth, when $\delta$ is held constant, we will often drop the subscript denoting
\[
\bfal \dfn \bfal_{\del}.
\]

With Assumption \ref{as:alpha} placed on $\bfal$, we will study the conditions that must be placed on the function $u$ so that $\cL_\bfal[u] \in L^2(\Omega)$ and $\mathcal{B}_{\bfal}[u] \in L^2(\Omega)$.  In particular, note that $\cL_\bfal$ is formally quadratic in $\bfal$ which means that the kernel $\mu$  in Definition \ref{def:laplace} would not be integrable for domains $\Om \subset \bbR^2$. The next few propositions will provide sufficient conditions which ensure that these functions are well-defined.

\begin{proposition}\label{cts} Suppose $\rho \in C^\infty_c(\mathbb{R}^d)$ and $d \geq 2$. If $\Omega \subset \mathbb{R}^d$ is bounded and $a>0$, then  the following mapping is continuous on $\bbR^d$:
\[ 
F(\bfx) \dfn \intOm \frac{\rho(|\bfy-\bfx|)}{|\bfy-\bfx|^{2-a}} d \bfy\,.
\] 
\end{proposition}
\begin{proof}
Fix $\bfx \in \mathbb{R}^d$. Note that the result clearly holds for $a\geq 2$, so assume $0<a<2$. Since $|\bfy-\bfx|^{-(2-a)} \in L^1(\mathbb{R}^d)$ for $a>0$ and $d \geq 2$, there exists an $r > 0$ such that
\begin{equation}\label{aux1}
\int_{B(\bfx,r)} \frac{1}{|\bfy-\bfx|^{2-a}} d \bfy < \frac{\eps}{3 \| \rho \|_{L^\infty} }\,.
\end{equation}
Since $\rho$ is smooth,
\[
\bfy\mapsto \kap(\bfx,\bfy)\dfn \frac{\rho(|\bfy-\bfx|)}{|\bfy-\bfx|^{2-a}} \in C\left(\mathbb{R}^d \setminus B \left(\bfx, \frac{r}{2} \right)\right).
\] 
Thus there exists a $\delta \in \left( 0 , \frac{r}{2} \right)$ such that whenever $\bfx' \in B(\bfx, \delta)$ we have 
\[
\left|
\kap(\bfx,\bfy) - \kap(\bfx',\bfy)\right| < \frac{\eps}{3 |\Omega| }
\quad\text{for all}\quad \bfy \in \Omega \setminus B(\bfx,r)\,.
\] 
Then
\begin{equation*}
\begin{split}
|F(\bfx)-F(\bfx')| =& \left| \intOm
\big[\kap(\bfx,\bfy) - \kap(\bfx',\bfy)\big] d \bfy \right| \\
\leq& \left| \int_{\Omega \setminus B(\bfx, \delta)} \big[\kap(\bfx,\bfy) - \kap(\bfx',\bfy)\big]  d \bfy \right| \\
&+ \|\rho\|_{L^\infty}\int_{B(\bfx, \delta)} \frac{1}{|\bfy-\bfx|^{2-a}} d \bfy + \|\rho\|_{L^\infty} \int_{B(\bfx, \delta)} \frac{1}{|\bfy-\bfx'|^{2-a}} d \bfy \\
\stackrel{\eqref{aux1}}{\leq}& \, \frac{\eps}{3 |\Omega|} \int_{\Omega \setminus B(\bfx,\delta)} d \bfy 
+ \|\rho\|_{L^\infty} \frac{\eps}{3 \|\rho\|_{L^\infty}} 
+ \|\rho\|_{L^\infty} \int_{B(\bfx', \delta)} \frac{1}{|\bfy-\bfx'|^{2-a}} d \bfy\\
\stackrel{\eqref{aux1}}{\leq}& \,\, \eps.
\end{split} 
\end{equation*}
The second to last line follows by noting that integration of $\frac{1}{|\bfy-\bfx'|^{2-a}}$ in terms of $\bfy$ over $B(\bfx', \delta)$ yields a larger value than integration over $B(\bfx,\delta)$, since the integrand is singular at $\bfy =\bfx'$.
It follows that  $F \in C(\mathbb{R}^d)$. 
\end{proof}

\begin{proposition}[$L^2$-integrability of the nonlocal Laplacian]\label{lapexists}
Suppose $\bfal$ satisfies Assumption \ref{as:alpha} and $d \geq 2$. If $\Omega \subset \mathbb{R}^d$ is a bounded Lipschitz domain and $u \in W^{1,p}(\Omega)$  with $p>2$ then $\cL_{\bfal}[u] \in L^2(\Omega)$.
\end{proposition}
\begin{proof}
For a shorthand let
\[
 \kap(\bfx,\bfy) \dfn \frac{\rho(|\bfx-\bfy|)}{|\bfy-\bfx|} \txtand
 q(\bfx,\bfy)\dfn \frac{u(\bfy)-\bf(x)}{|\bfy-\bfx|}.
\]
Also, henceforth, if $\psi(\bfx,\bfy)$  is a function on $\Om\times \Om$ then $L^q(\Om, \bfx)$ or $L^q(\Om, \bfy)$  will denote the $L^q$ norm of the $\bfy$- or $\bfx$-section of $\psi$ respectively. 

By H\"{o}lder's inequality, for any $\eps\in(0,1)$ we have

\begin{equation*}
\begin{split}
\|\cL_\bfal[\phi]\|_{L^2}^2 &= \intOm \left| \intOm \frac{u(\bfy)-u(\bfx)}{|\bfy-\bfx|^2} \rho(\bfx,\bfy) d \bfy \right|^2 d \bfx \\
&\leq \intOm \left( \left\| q(\bfx,\bfy)\right\|_{L^{p}(\Omega,\bfy)} \left\| \kap(\bfx,\bfy) \right\|_{L^{p^*}(\Omega,\bfy)} \right)^2 d \bfx\,.
\end{split}
\end{equation*}
where $p = \frac{2-\varepsilon}{1-\varepsilon}$ and $p^*$ is the H\"{o}lder conjugate of $p$.

Another application of H\"{o}lder's inequality to the integral in $\bfx$ yields: 
\begin{equation*}
\begin{split}
\|\cL_\bfal[u]\|_{L^2}^2 &\leq \intOm   \left\|  q(\bfx,\bfy) \right\|^2_{L^{p}(\Omega,\bfy)}   \left\| \kap(\bfx,\bfy) \right\|^2_{L^{p^*}(\Omega,\bfy)}    d \bfx \\
&\leq \left\| \;  \left\| q(\bfx,\bfy) \right\|^2_{L^{p}(\Omega,\bfy)}  \right\|_{L^{\frac{p}{2}}(\Omega,\bfx)} 
\left\| \;  \left\| \kap(\bfx,\bfy) \right\|^2_{L^{p^*}(\Omega,\bfy)} \right\|_{L^{\frac{p^*}{\eps}}(\Omega,\bfx)} \\
&= \left( \intOm \intOm |q(\bfx,\bfy)|^{p} d \bfy d\bfx \right)^{\frac{2}{p}} \left[ \intOm \left( \intOm \kap(\bfx,\bfy)^{p^*} d\bfy \right)^{\frac{2}{\eps}} d \bfx \right]^{\frac{\eps}{p^*}}.
\end{split}
\end{equation*}
Applying Theorem \ref{bourgain2}  (for instance with $\xi = \chi_{\Om}$) to the first integral factor to obtain: 
\begin{equation*}
\begin{split}
\|\cL_\bfal[\phi]\|_{\ltom}^2 & \lesssim \left(  \| u \|_{W^{1,p}}^p \right)^{\frac{2}{p}} \bigg[ \intOm \left( \intOm \bigg( \frac{\rho(|\bfx-\bfy|)}{|\bfy-\bfx|} \bigg)^{p^*} d\bfy \right)^{\frac{2}{\eps}} d \bfx \bigg]^{\frac{\eps}{p^*}}.
\end{split}
\end{equation*}
Proposition \ref{cts} implies that 
\[
\intOm \left( \frac{\rho(|\bfx-\bfy|)}{|\bfy-\bfx|} \right)^{p^*} d\bfy  
\]
is bounded on $\overline{\Omega}$ (notice $p^* = 2-\varepsilon$); consequently, $\| \cL_\bfal[u] \|_{L^2(\Omega)} < \infty$ provided $u \in W^{1,p}(\Omega)$. As $\eps\searrow 0$ the integrability index $p$ tends to $2$ from above, hence the condition $p>2$.


\end{proof}

\begin{theorem}[Lipschitz continuity of the nonlocal Laplacian]\label{lapholder} 
Suppose $\bfal$ satisfies Assumption \ref{as:alpha}, $\Omega \subset \mathbb{R}^d$ is a bounded open set, and $d \geq 2$. If $u \in C^2(\Omega)\cap W^{2,\infty}(\Omega)$, then $\bfx \to \cL_\bfal[u](\bfx)$ is Lipschitz continuous. 
\end{theorem}

\proof
From the assumption
$u \in C^2(\Omega)\cap W^{2,\infty}(\Omega)$
 it follows that $u$ is Lipschitz on $\Om$, hence $y\mapsto (u(\bfy)-\bfy(x))/|\bfy-\bfx|^2$ is integrable on $\Om$, provided the space dimension is $d \geq 2$. 
Use the fact that a fortiori $u \in C^1(\Omega)$ with support compactly contained in $\Omega$ to define
\begin{equation*}
 A(\bfx,\bfy)\dfn 
 \int_0^1 \grad u(\lambda \bfy +(1-\lambda)\bfx)d \lambda .
\end{equation*}
From the fundamental theorem of calculus we have
\begin{equation}\label{eqn:ftc}
u(\bfy)-u(\bfx)= \int_0^1 \grad u(\lambda \bfy +(1-\lambda)\bfx) \cdot (\bfy-\bfx) d \lambda = A(\bfx,\bfy) \cdot (\bfy-\bfx)
\end{equation}
Recalling the definition of the nonlocal Laplacian we have
\begin{equation*}
\begin{split}
\cL_\bfal[u](\bfx)-\cL_\bfal[u](\bfx') =  \intOm \frac{(u(\bfy)-u(\bfx))\rho(|\bfy-\bfx|)}{|\bfy-\bfx|^2}- \frac{(u(\bfy)-u(\bfx'))\rho(|\bfy-\bfx'|)}{|\bfy-\bfx'|^2} d \bfy \\
\end{split}
\end{equation*}
Utilizing equation \eqref{eqn:ftc} results in
\begin{equation*}
\begin{split}
&\cL_\bfal[u](\bfx)-\cL_\bfal[u](\bfx') = \\
=&  \intOm \frac{A(\bfx,\bfy)\cdot(\bfy-\bfx)}{|\bfy-\bfx|^{2}}  \rho(|\bfx-\bfy|)- \frac{A(\bfx',\bfy)\cdot(\bfy-\bfx')}{|\bfy-\bfx'|^{2}}  \rho(|\bfx'-\bfy|) d \bfy  \\ 
=& \int_{B_{\delta}(\bfx)} \frac{A(\bfx,\bfy)\cdot(\bfy-\bfx)}{|\bfy-\bfx|^{2}}  \rho(|\bfx-\bfy|)d \bfy 
-\int_{B_{\delta}(\bfx')}  \frac{A(\bfx',\bfy)\cdot(\bfy-\bfx')}{|\bfy-\bfx'|^{2}}  \rho(|\bfx'-\bfy|)  d\bfy.
\end{split}
\end{equation*}
Use the fact that $\supp(\rho_\del)\subset [0,\del)$ and change variables so the center of each of the two balls is at the origin to get
\begin{equation}\label{eqn:lapineq}
\cL_\bfal[u](\bfx)-\cL_\bfal[u](\bfx') = \int_{B_{\delta}(0)}\frac{(A(\bfx,\bfx+\bfz)-A(\bfx',\bfx'+\bfz))\cdot \bfz}{|\bfz|^{2}}  \rho(|\bfz|) d\bfz
\end{equation}

Due to the assumption $u \in C^2(\Omega) \cap W^{2,\infty}(\Omega)$, we know $A\in C^1(\Omega\times \Omega)\cap W^{1,\infty}(\Omega \times \Omega)$. It follows that
\[ 
|A(\bfx,\bfx-\bfz)-A(\bfx',\bfx'+\bfz)| \leq M|\bfx-\bfx'|
\]
for some constant $M$ dependent on the $L^\infty$-norm of the second-order derivatives of $u$. Combining this with \eqref{eqn:lapineq} results in
 
\begin{equation*}
\begin{split}
\left| \mathcal{L}_{\bfal}[u](\bfx)-\mathcal{L}_{\bfal}[u](\bfx') \right| \leq  M|\bfx-\bfx'| \int_{B_\delta(0)} \frac{\rho(|\bfz|)}{|\bfz|} d \bfz = N |\bfx-\bfx'|. 
\end{split}
\end{equation*}
where $\displaystyle N = M \int_{B_\delta(0)} \frac{ \rho(|\bfz|)}{|\bfz|} d \bfz < \infty$ is a constant.
\qed

\begin{proposition}[Regularity of the nonlocal biharmonic]\label{ctslap}
Let $\eps \in (0,1)$, $d \geq 2$ and $\Om\subset \bbR^d$ be a bounded open set. Suppose $u \in W^{2,\infty}(\Omega) \cap C^2(\Omega)$. Furthermore, let $\rho \in C^\infty_c(\mathbb{R}^+)$. Then $\mathcal{B}_\bfal[u] \in L^\infty(\Omega)$ (hence in $\ltom$). 
\end{proposition}

\proof
Notice for $\bfx \in \Omega$ we have
\[
\mathcal{B}_\bfal[u](\bfx) = \intOm \frac{(\cL_\bfal[u] (\bfx')-\cL_{\bfal}[u](\bfx))\rho(|\bfx'-\bfx|)}{|\bfx' - \bfx|^2} d \bfx'.
\]

By appealing to Theorem \ref{lapholder} we know $\cL_\bfal[u]$ is Lipschitz continuous on $\Omega$, and consequently there exists a constant $M$ such that
\[
|\mathcal{B}_\bfal [u](\bfx)| \leq M \intOm \frac{\rho(|\bfx'-\bfx|)}{|\bfx' - \bfx|} d \bfx' < \infty.
\]
Since $\Omega \subset \mathbb{R}^2$ is bounded we conclude $|\mathcal{B}[u](\bfx)|$ is essentially bounded on $\Omega$. 

\qed

\section{Nonlocal function spaces} This section will introduce various Hilbert spaces we will be working in for the formulation of our nonlocal problems later in the paper.
Following \cite{DuMengesha} we will utilize the functional space  
\begin{equation}
\mathscr{H}_\bfal^1(\Omega) := \left\{ u \in L^2(\Omega) :  \| \cG_{\bfal}[u] \|_{L^2(\Om\times \Om)} < \infty \right\}.
\end{equation}
Define the bilinear forms
\begin{equation}
\lp u,w\rp_1 = \intOm \intOm \cG_{\bfal}[u]\cdot\cG_{\bfal}[w] d\bfx' d\bfx
\end{equation}
and 
\begin{equation}\label{H^1-productd}
(u,w)_{\mathscr{H}^1_\bfal} = (u,w)_{L^2(\Om)} + \lp u,w\rp_1\,.
\end{equation}
Note that if $|\bfal|^2$ is integrable, then $\scrH^1_\bfal(\Om)$ is equivalent to $\ltom$. However, under Assumption \ref{as:alpha}, this may not be the case when $\Omega \subset \mathbb{R}^2$.

\begin{theorem}[{c.f. \cite[Thm 2.2]{DuMengesha}}] Assume $\bfal$ satisfies Assumption \ref{as:alpha}. Then, $\scrH_\bfal^1(\Om)$ is a Hilbert space with inner product \eqref{H^1-productd}.
\end{theorem}

\begin{definition} For $\Omega'\subset \subset\Omega$, define $\mathscr{H}^1_0(\Omega')$ to be the closed subspace of functions vanishing on $\Om\setminus \Om'$
\begin{equation*}
\mathscr{H}^1_{\bfal,0}(\Omega') = \left\{ u \in \mathscr{H}_\bfal^1(\Omega):  u=0 \text{ a.e. in } \Om\setminus \Om'\right\}\,.
\end{equation*}  
\end{definition}

%
\begin{definition}\label{def:nonlocalweaklap}
Let $u \in \mathscr{H}^1_\bfal(\Omega)$. We say $v \in L^2(\Omega)$ is the \textbf{nonlocal weak} Laplacian of $u$ provided
\[
-\lp u , \phi \rp_1 =  \intOm v \phi d \bfx \quad \forall \phi \in \mathscr{H}_\bfal^1 (\Omega).
\]
\end{definition}

\begin{proposition}\label{prop:weak-strong}
Let $\Omega \subset \mathbb{R}^d$ be open, $u : \Omega \to \mathbb{R}$, $\bfal:\Omega \times \Omega \to \mathbb{R}^k$ be antisymmetric. If $\mathcal{L}_\bfal[u] \in L^2(\Omega)$ and $\cG_\bfal[u] \in L^2(\Omega^2)$ then the weak nonlocal Laplacian $\mathcal{L}^*_\bfal[u]$ and the nonlocal Laplacian $\mathcal{L}_\bfal[u]$ agree a.e. in $\Omega$.
\end{proposition}
\begin{proof}
Let $v \in C^\infty_c(\Omega)$. Then by Definition \ref{def:nonlocalweaklap} and Proposition \ref{parts} we have
\begin{equation*}
\int_\Omega \mathcal{L}^*[u] v d \bfx = - \int_\Omega \int_\Omega \cG_\bfal[u]\cdot\cG_\bfal[v] d \bfy d \bfx = \int_\Omega \mathcal{L}[u] v d \bfx.
\end{equation*}
Since this holds for any $v \in C^\infty_c(\Omega)$ we conclude $\mathcal{L}^*[u] = \mathcal{L}[u]$ a.e. in $\Omega$.
\end{proof}

\begin{remark}\label{rem:by-parts-weak-strong}
For the remainder of the paper, we will use the notation $\cL_\bfal[u]$ to denote the weak Laplacian of the function $u$. When the distinction between weak and original definition is essential it will be indicated. To begin with, note that the integration by parts Proposition \ref{parts} holds for the nonlocal weak Laplacian simply by definition of the latter.
\end{remark}

\begin{definition}  Let
\begin{equation}
\mathscr{H}_\bfal^2(\Omega) := \left\{ u \in \mathscr{H}_\bfal^1(\Omega) :  \cL_\bfal[u] \in L^2(\Omega) \right\}
\end{equation}
and defining
\begin{equation}
\lp u,w\rp_2 \dfn \intOm \intOm \cL_\bfal [u] \cL_\bfal [w] d\bfx' d\bfx\,
\end{equation}
introduce an inner product on $\scrH^2_\bfal$:
\[
(u,w)_{\scrH^2_\bfal}=  (u,w)_{\scrH^1_\bfal} + \lp u, w\rp_2\,.
\]
\end{definition}

\begin{proposition}\label{H1complete} Suppose $\boldsymbol{\alpha}$ satisfies Assumption \ref{as:alpha}. The space $\scrH^1_\bfal(\Omega)$ is a Hilbert space with inner product $(\cdot, \cdot)_{\scrH^1}$.
\end{proposition}
\begin{proof}
All that remains to be proven is completeness. Let $(u_n)$ be a Cauchy sequence in $\scrH^1_\bfal(\Omega)$. By definition, $(u_n)$ is a Cauchy sequence in $L^2(\Omega)$ and consequently there exists a strong $\ltom$-limit $u \in L^2(\Omega)$. Similarly, we also know $\cG_\bfal [u_n]$ converges to some $L^2$-norm limit $\mathbf{v} \in L^2(\Omega\times \Omega)$. Consider the truncated kernel $\boldsymbol{\alpha}_\tau = \boldsymbol{\alpha} \chi_{\left[ \frac{1}{\tau}, \infty \right)}$ for $\tau > 0$. 
For every $\tau > 0$ and $\boldsymbol{\phi} \in L^2(\Omega \times \Omega)$ we have by proposition \ref{parts},
\begin{equation*}
\begin{split}
\lim_{n \to \infty} \left| \intOm \intOm \left(\cG_{\bfal_\tau} [u_n] - \cG_{\bfal_\tau}[u]\right) \cdot \boldsymbol{\phi} dy dx \right| &= \lim_{n \to \infty} \left| \intOm \intOm \cG_{\bfal_\tau} [u_n-u] \cdot \boldsymbol{\phi} dy dx \right| \\
&= \lim_{n \to \infty} \left| \intOm (u_n-u) \mathcal{D}_{\bfal_\tau}[\boldsymbol{\phi}] dx \right| \\
&\leq \lim_{n \to \infty} \| u_n -u \|_{L^2} \| \mathcal{D}_{\bfal_\tau}[\boldsymbol{\phi}] \|_{L^2} \\
&=0.
\end{split}
\end{equation*}
Note that $\mathcal{D}_{\bfal_\tau}[\mathbf{v}] \in L^2(\Omega)$ since $(\mathbf{v}'+\mathbf{v})\bfal_\tau \in C^\infty_c(\Omega \times \Omega)$ for every $\tau > 0$. We conclude $\lim\limits_{n \to \infty} \cG_{\bfal_\tau}[u_n] = \cG_{\bfal_\tau}[u]$ a.e. for $(\bfx,\bfy) \in \Omega \times \Omega$. Combining this with the fact that $\cG_\bfal[u_n] \to \bfv$ in $L^2(\Omega \times \Omega)$ implies 
for every $\tau > 0$, 
\[
\cG_{\bfal_\tau}[u] = \lim_{n \to \infty} \cG_{\bfal_\tau}[u_n] = \lim_{n \to \infty} \cG_{\bfal}[u_n] = \mathbf{v} \quad \text{a.e. in} \; \left\{(\bfx,\bfy) \in \Omega \times \Omega : |\bfx - \bfy| \geq \frac{1}{\tau} \right\}.
\] 
From this we conclude $\cG_{\bfal}[u] = \mathbf{v}$ a.e. in $\Omega \time \Omega \times \Omega$.

\end{proof}

\begin{proposition}\label{H2complete} The space $\scrH^2_\alpha (\Omega)$ is a Hilbert space with inner product $(\cdot, \cdot)_{\scrH^2_\alpha}$.
\end{proposition}
\begin{proof} 

Again we only need to verify completeness.
Let $(u_n)$ be a Cauchy sequence in $\scrH^2_\alpha(\Omega)$. Then $(u_n)$, $(\cG[u_n])$, and $(\cL_{\alpha}(u_n))$ are Cauchy in $L^2$ and consequently have $L^2$ limits $u,v,$ and $w$ respectively. Using the same method as in Proposition \ref{H1complete}, we have $v = \cG[u_n]$ a.e. Let $\phi \in \mathscr{H}^1(\Omega)$. 

Appealing to Proposition \ref{parts}, Theorem \ref{bourgain2}, and H\"{o}lder's inequality, we obtain
\begin{equation*}
\begin{split}
\lim_{n \to \infty} \intOm \intOm \cG_{\bfal}[u_n-u]\cdot\cG_{\bfal}[\phi] d \bfy d \bfx &\leq \lim_{n \to \infty} \| \cG_\bfal[u-u_n] \|_{L^2(\Omega \times \Omega)} \| \cG_{\bfal}[\phi] \|_{L^2(\Omega \times \Omega)} \\
&\leq  \lim_{n \to \infty} C \| \cG_\bfal[u-u_n] \|_{L^2(\Omega \times \Omega)} \| \phi \|_{W^{1,2}(\Omega)} \\
&= 0
\end{split}
\end{equation*}
and 
\begin{equation*}
\begin{split}
\lim_{n \to \infty} \intOm \intOm (w-\cL_\bfal[u_n]) \phi d \bfy d \bfx &\leq \lim_{n \to \infty} \| w-\cL_\bfal [u_n] \|_{L^2(\Omega)} \| \phi \|_{L^2(\Omega)} \\
&= 0 .
\end{split}
\end{equation*}

Putting this together, we obtain via Proposition \ref{parts}
\[
\begin{split}
- \intOm \intOm \cG_{\bfal}[u]\cdot\cG_{\bfal} [\phi] d \bfx =& - \lim_{n \to \infty} \intOm \intOm \cG_\bfal [u_n]\cdot\cG_\bfal [\phi] d \bfx \\
=& \lim_{n \to \infty} \intOm \cL_\bfal [u_n] \phi d \bfx = \intOm w \phi dx\,.
\end{split}
\]
Thus, $w = \cL_\bfal [u]$ a.e. in $\Omega$ in the nonlocal weak sense and, consequently, $u \in \scrH^2_\bfal(\Omega)$.
\end{proof}


Finally, we  define Hilbert spaces associated with the boundary conditions that we will consider.

\begin{definition}[Nonlocal  ``hinged" and ``clamped" spaces]\label{def:hinged-clamped}
Let $\Om''\subset\subset \Om'$ be an open subdomain of $\Om'$. Define respectively
\begin{equation}\label{def:H-hinged}
\mathscr{H}_{\bfal, H}^2 (\Omega',\Om'') = \left\{ u \in \mathscr{H}_{\bfal,0}^1(\Omega') \cap \scrH_{\bfal}^2(\Om) :  \cL_{\bfal}[u] = 0 \text{ a.e. on }\Om' \setminus \Om''\right\} 
\end{equation}
\begin{equation}\label{def:H-clamped}
\begin{split}
\mathscr{H}_{\bfal, C}^2 (\Omega') = \bigg\{ & u \in \mathscr{H}_{\bfal,0}^1(\Omega')\cap  \scrH_{\bfal}^2(\Om):   \mathcal{N}_\bfal[\cG_\bfal[u]] = 0 \text{ a.e. on } \inter(\Omega') \bigg\} .
\end{split}
\end{equation}
\end{definition}

From the definition \eqref{def:nonlocal-normal} of the nonlocal normal operator and the nonlocal two-point gradient \eqref{def:nonlocal-grad}  it follows that
\[
\cN_{\bfal}[\cG_{\bfal}[u]](\bfx) = \int_{\Om\setminus \Om'} (u(\bfy) - u(\bfx)) \mu(\bfx,\bfy) d\bfy \txtfor \bfx\in \inter(\Om')\,.
\]
Moreover, because in the clamped space we also have $u(\bfy)=0$ on $\Om\setminus \Om'$, then the identity $\cN_{\bfal}[\cG_{\bfal}[u]](\bfx)=0$ actually reduces to
\[
-u(\bfx)\int_{\Om\setminus \Om'} \mu(\bfx,\bfy) d\bfy = 0\,.
\]
If we set $\Om' = \Om_{\delta}$ where
\[
 \Omega_\delta \dfn \left\{ \bfx \in \Omega : \text{dist}(\bfx,\partial \Omega) > \delta \right\}
 \]
 and choose $\bfal=\bfal_{\del}$ as in Assumption \ref{as:alpha}, 
then  because $y\mapsto \mu(\bfx,\bfy)$ is  strictly positive and continuous on $\Om\setminus \Om'$ for any fixed $\bfx \in \inter\Om_{\del}$, we conclude that $u(\bfx) =0$ a.e. in $\Om\setminus \Om_{2\del}$. Thus we have an alternative representation for the  nonlocal ``clamped" space:
\begin{equation}\label{H-clamped-delta}
\mathscr{H}_{\bfal_{\del}, C}^2 (\Omega_{\del}) =  \mathscr{H}_{\bfal_{(2\del)},0}^1(\Omega_{\del})\cap  \scrH_{\bfal_{\del}}^2(\Om)\,.
\end{equation}

\section{Compactness theorems}

A key tool in subsequent analysis will be the following version of a nonlocal Poincar\'e inequality. 

\begin{theorem}[{\cite[Thm. 1.2]{Ponce}}]\label{Poincare} Assume $\Om$ is a bounded domain  of dimension  $d=\dim\Om \geq 2$ 
with  Lipschitz boundary.
 Let $(\delta_n)$ be a sequence of positive numbers decreasing to $0$. Let $(\rho_{\delta_n})$ be a sequence of functions satisfying \eqref{mollifier}. There exists  a constant $C_{p,d,\Om}$ dependent on the domain, $p$  (and also on the choice of the sequence of mollifiers $\rho_{\del_n}$) such that
\begin{equation*}
\| f - f_\Omega \|_{L^p(\Omega)}^p \leq C_{p,d, \Omega} \intOm \intOm \frac{|f(\bfx)-f(\bfy)|^p}{|\bfx-\bfy|^p} \rho_{\delta_n}(|\bfx-\bfy|) d\bfx d\bfy
\end{equation*}
for every $f \in L^p(\Omega)$ with the convention that the right-hand side is $+\infty$ when diverges. Here $f_\Omega$ is the average value of $f$ in $\Omega$.
\end{theorem}
\begin{remark}
It should be noted that in \cite{Ponce}, the result of Theorem \ref{Poincare}  is extended to dimension $d = 1$; however, in that case, it is necessary to place an additional constraint on $\rho_{\delta_n}$. 
\end{remark}


In the nonlocal setting we cannot appeal to the embedding and compactness methods of Sobolev theory.
Instead, the crucial compactness result in this context will be provided by the following theorem of Brezis, Bourgain, and Mironescu: 
\begin{theorem}[{\cite[Thm. 4]{Bourgain}}]\label{thm:UniformGrad} 
Let $\Om$ be a bounded domain of class $C^1$. Let $(\delta_n)$ be a sequence decreasing to $0$. Suppose $(f_n)$ is a sequence in $L^p(\Om)$, $1\leq p <\infty$, of functions satisfying the uniform estimate
\begin{equation}\label{bound-on-quotient}
\intOm \intOm \frac{|f_n(\bfx)-f_n(\bfy)|^p}{|\bfx-\bfy|^p} \rho_{\delta_n}(|\bfx-\bfy|) d\bfx d\bfy \leq C_0
\end{equation}
where $(\rho_{\delta_n})$ is a sequence of non-increasing mollifiers satisfying \eqref{mollifier}. If
\begin{equation}\label{zero-average}
\intOm f_n(\bfx) d\bfx =0 \txtforall n,
\end{equation}
then 
\begin{enumerate}[(i)]
\item
the sequence $(f_n)$ is relatively compact in $L^p(\Om)$, so up to a subsequence we may assume $f_n \to f$ in $L^p(\Om)$

\item if, in addition, $1 < p < \infty$, then $f \in W^{1,p}(\Om)$ and $\|\nabla f\|_{L^p(\Om)}\leq K(p,d) C_0$ for $K$ dependent only on $p$ and the dimension $d$. 
\end{enumerate}
\end{theorem}

\begin{remark}\label{rem:regularity}
The compactness in $L^p$ (part (i)) result of \cite[Thm. 4]{Bourgain} uses Riesz-Fr\'echet-Kolmogorov's theorem (e.g., see \cite[Thm. IV.25, p. 72]{b:brezis:83}) which establishes this  compactness result on a set compactly contained within a given open domain. To get the conclusion on all of $\Om$ the proof of   \cite[Thm. 4]{Bourgain}  uses an extension of the functions $f_n$ by reflection across the boundary of $\Om$;  due to the monotonicity condition on $\rho_{\del_n}$ such a reflection preserves the property \eqref{bound-on-quotient}. 
Thus, part (i) needs the $C^1$ regularity  only to obtain an $L^p$-preserving extension by reflection. This result  reduces to a change of variable theorem for the mapping that locally defines the boundary; this procedure  could potentially be carried out under weaker boundary regularity conditions (though stronger than Lipschitz), for instance, see \cite{hajlasz:93:cm}.
\end{remark}

Recall that
 \[
 \Omega_\delta \dfn \left\{ \bfx \in \Omega : \text{dist}(\bfx,\partial \Omega) > \delta \right\}\,.
\] 
Below we present a useful corollary to Theorem \ref{thm:UniformGrad}. This result is also mentioned in \cite{DuMengesha}. For a self contained exposition we provide a proof.
\begin{coro}\label{coro:UniformGrad}
In Theorem \ref{thm:UniformGrad}  we can replace assumption \eqref{zero-average} by the assertion that $(f_n)$ are bounded in $L^p(\Om)$. Moreover, if  $1<p<\infty$ and $\supp f_n \subset \cl{\Om}_{\delta_n}$, then $f\in W^{1,p}_0(\Om)$. 
\end{coro}
\begin{proof}
Let $\displaystyle a_n \dfn \frac{1}{|\Om|}\intOm f_n(\bfx) d\bfx$. Since $\Om$ is bounded and $(f_n)$ is bounded in $L^p(\Om)$, then the scalar sequence $(a_n)$ is bounded.  Define
\[
  g_n(\bfx) \dfn f_n(\bfx) - a_n\,.
\]
Then each $g_n$ obeys \eqref{bound-on-quotient} and has zero average. By Theorem \ref{thm:UniformGrad} we know $\{g_n\}$ is pre-compact in $L^p(\Om)$. Because $(a_n)$ is a bounded scalar sequence and $\Om$ is bounded, then $\{f_n\}$  is also pre-compact in $L^p(\Om)$.

Now to show $f \in W^{1,p}_0(\Omega)$. 
Recall supp $f_n \subset \subset \Omega_{\delta_n}$. Choose $\Omega_{\text{ext}} \subset \mathbb{R}^d$ open and bounded such that $\Omega \subset \subset \Omega_{\text{ext}}$. Let
\begin{equation*}
\tilde{f}_n = \left\{ 
\begin{array}{ll}
f_n, & \bfx \in \Omega \\
0, & \bfx \in \Omega_{\text{ext}} \setminus \Omega
\end{array}
\right.
\end{equation*}
Notice that the integral in
\[
\int_{\Omega_{\text{ext}}} \int_{\Omega_{\text{ext}}} \frac{|\tilde{f}_n(\bfx)-\tilde{f}_n(\bfy)|^p}{|\bfx-\bfy|^p} \rho_{\delta_n}(|\bfx-\bfy|) d\bfx d\bfy
\]
can be decomposed as
\begin{equation*}
\begin{split}
\left[ 
\intOm \intOm
+ 
\int_{\Omega_{\text{ext}}\setminus \Om  }
\intOm
+ 
\intOm
\int_{\Omega_{\text{ext}} \setminus \Omega_{\delta_n}} \right](\ldots) d\bfx d\bfy 
\end{split}
\end{equation*}
When $\bfx\in \Om_{\ext}\setminus \Om$ then either $\bfy\in \Om \setminus \Om_{\del_n}$ in which case $\tilde{f_n}(\bfx) - \tilde{f}_n(\bfy) = 0$ via the zero condition on the collar possibly excluding a set of measure zero, or the distance between $\bfx$ and $\bfy$ is at least $\del$ whence $\rho_{\del_n}(|\bfx-\bfy|)=0$. So the second double integral above is $0$.  A symmetric argument shows that the third one is zero as well. Conclude:
\begin{equation*}
\begin{split}
\int_{\Omega_{\text{ext}}} \int_{\Omega_{\text{ext}}} \frac{|\tilde{f}_n(\bfx)-\tilde{f}_n(\bfy)|^p}{|\bfx-\bfy|^p} \rho_{\delta_n} d\bfx d\bfy =& \int_{\Omega} \int_{\Omega} \frac{|f_n(\bfx)-f_n(\bfy)|^p}{|\bfx-\bfy|^p} \rho_{\delta_n} d\bfx d\bfy \leq C_0
\end{split}
\end{equation*}
By Theorem \ref{thm:UniformGrad} we know $\tilde{f} := \lim\limits_{n \to \infty} \tilde{f}_n \in W^{1,p}(\Omega_{\text{ext}})$. 
Notice that $\tilde{f}$ is a $W^{1,2}$ extension of $f$ to $\Omega_{\text{ext}}$ across $\partial \Omega$. The trace of $\tl{f}$ on the ($C^1$) sub-manifold $\p\Om$ is then uniquely determined  (e.g., \cite[Thm. 5.36]{adams}) by treating it as the boundary of $\Om$ and of $\Om_{\ext}\setminus\Om $ respectively.  Since $\tl{f}=0$ on $\Om_{\ext}\setminus \Om $ we conclude $f = 0$ on $\partial \Omega$.
\end{proof}

The preceding Corollary \ref{coro:UniformGrad}, in turn allows us to state two versions of Theorem \ref{Poincare} applicable to functions compactly supported in $\Om$:

\begin{coro}[Nonlocal Poincar\'e with zero shrinking collar for a sequence]\label{coro:poincare:0}
Under the assumptions of Theorem \ref{Poincare}, suppose $(f_n)$ is a family of functions in $L^p(\Om)$ such that $\supp f_n \subset \cl{\Om}_{\del_n}$. Then there is a constant $C_{p,d,\Om}>0$ satisfying
\begin{equation*}
\| f_n \|_{L^p(\Omega)}^p \leq C_{p,d, \Omega} \intOm \intOm \frac{|f_n(\bfx)-f_n(\bfy)|^p}{|\bfx-\bfy|^p} \rho_{\delta_n}(|\bfx-\bfy|) d\bfx d\bfy
\end{equation*}
for every $n$ with the convention that the right-hand side is $+\infty $ when undefined.
\end{coro}
\begin{proof} 
Proceed by contradiction as in \cite[p. 12]{Ponce}. 
Suppose we can extract a subsequence (re-indexed again by $n$) so that the candidates for ``$C_{p,d,\Om}$" diverge to $+\infty$. In particular, suppose a sequence of scalars $(c_n)$ diverges to $+\infty$  and for every $n$
\[
\|f_n\|^p_{L^p(\Om)} \geq c_n  \intOm \intOm \frac{|f(\bfx)- f(\bfy)|^p}{|\bfx-\bfy|^p}\rho_{\del_n}(|\bfx-\bfy|) d\bfx d\bfy 
\]
Re-normalizing both sides by $\|f_n\|^p_{L^p(\Om)}$  we may just assume that  $\|f_n\|_{L^p}=1$ and 
\begin{equation}\label{zero-limit}
 \lim_{n\to \infty}\intOm \intOm \frac{|f_n(\bfx)- f_n(\bfy)|^p}{|\bfx-\bfy|^p}\rho_{\del_n}(|\bfx-\bfy|) d\bfx d\bfy =0\,. 
\end{equation}
By Corollary \ref{coro:UniformGrad}  functions $(f_n)$ converge strongly to $f$ in $L^p(\Om)$ with $\|f\|_{L^p(\Om)}=1$. Moreover, $f\in W^{1,p}_0(\Om)$.
In addition, by and the same Corollary (referring now to part (ii) of Theorem \ref{thm:UniformGrad}) we have from \eqref{zero-limit} that $\|\nabla f\|_{L^p(\Om)}=0$.  Poincar\'e-Wirtinger's inequality now implies that $f =0$  in $W^{1,p}_0(\Om)$ thus contradicting the fact that $\|f\|_{L^p(\Om)}=1$.
\end{proof}

\begin{coro}[Nonlocal Poincar\'e with zero collar]\label{coro:poincare}
Under the assumptions of Theorem \ref{Poincare}, suppose if $\Om_0 \subset\subset\Om$. Then there is a constant  $C_{p,d,\Om, \Om_0}>0$ such that 
\begin{equation*}
\| f \|_{L^p(\Omega)}^p \leq C_{p,d, \Omega, \Om_0} \intOm \intOm \frac{|f(\bfx)-f(\bfy)|^p}{|\bfx-\bfy|^p} \rho_{\delta_n}(|\bfx-\bfy|) d\bfx d\bfy
\end{equation*}
for all $f\in L^p(\Om)$ satisfying  $\supp f \subset \Om_0$.
\end{coro}
\begin{proof}
The proof by contradiction reduces to a construction of a sequence $(f_n)$ in $L^p(\Om)$ that violates  Corollary \ref{coro:poincare:0}. Since $\Om_0\subset \subset \Om$ is fixed, then we may assume that each $\Om_0 \subseteq \Om_{\del_n}$ for all $n$. 
\end{proof}


\section{Convergence of the nonlocal operators}

\subsection{Scaled operators}
With appropriate scaling, the nonlocal Laplace and biharmonic operators converge to their local analogues.  Throughout this subsection suppose Assumption \ref{as:alpha} holds for kernel $\bfal_\del$ and 
\[
\mu_\delta(\bfx,\bfy) \dfn |\bfal_\delta(\bfx,\bfy)|^2 = \frac{\rho_\del(|\bfx-\bfy|)}{|\bfx-\bfy|^2}\,.
\] 
Since this is a radial function we may write with a slight abuse of notation:
\[
 \mu_\del(\bfx,\bfy) = \mu_\del(|\bfx-\bfy|) \txtand    \mu_\del(s) = \frac{\rho_\del(s)}{s^2}.
\]

\begin{definition}[Scaling] Let
\begin{equation}\label{pi}
\pi_\del(r) := \int_r^{\delta}s \mu_{\delta}(s)  d s\,.
\end{equation}
Let $\om_{d-1}$ be the surface measure of unit sphere in $\bbR^d$ and define
\begin{equation}\label{Cdelta}
C(\delta) \dfn 
\half \int_{B_\delta(0)}   \pi(|\mathbf{z}|) d\mathbf{z} =
\frac{1}{2} \omega_{d-1} \int_0^\delta \pi(r) r^{d-1} d r,
\end{equation}
which is finite for $d \geq 2$. 
\end{definition}

\begin{proposition}\label{prop:Cdelta}
Let $\delta > 0$ and $C(\delta)$ be given by \eqref{Cdelta}. Then 
\[
C(\delta) = 1/(2d)
\]
\end{proposition}
\begin{proof}
By appealing to the definitions of $C(\delta),\pi_\delta(r), \mu_\delta(s)$ and changing the order of integration, we obtain 
\begin{equation*}
\begin{split}
C(\delta) =& \frac{\omega_{d-1}}{2} \int_0^\delta \int_r^\delta \frac{\rho_\delta(s)}{s} r^{d-1} ds dr 
= \frac{\omega_{d-1}}{2} \int_0^\delta \int_0^s r^{d-1} \frac{\rho_\delta(s)}{s} dr ds \\
=& \frac{\omega_{d-1}}{2d} \int_0^\delta s^{d-1} \rho_\delta(s) ds\,.
\end{split}
\end{equation*}
The result follows since by construction $\int_{\bbR^d}\rho_\del(\bfx) d\bfx = \int_{B_\del(0)} \rho_\delta(\bfx) d\bfx = 1$.
\end{proof}
In light of the previous proposition we can write $C$ rather than $C(\delta)$ in \eqref{Cdelta} and introduce
\begin{equation}\label{def:scaling}
\sig \dfn C^{-1} = 2d\,.
\end{equation}
Accordingly, 
we  redefine the nonlocal Laplacian and biharmonic operators with the scaling term:
\begin{eqnarray}\label{scaledop}
\cL_{\bfal_\delta} u(\bfx) &\dfn& \sigma \intOm [u(\bfy)-u(\bfx)] \mu_\delta (\bfx,\bfy) d\bfy \\
\mathcal{B}_{\bfal_\delta}u(\bfx) &\dfn & \cL_{\bfal_\delta} \left[ \cL_{\bfal_\delta} [u] (\bfx) \right] .
\end{eqnarray}

\subsection{Pointwise convergence}

We will show that the nonlocal operators approach uniformly their classical versions when acting on smooth functions as the peridynamic horizon $\delta$ goes to zero.  The proofs of the following results were inspired by the strategy used in the upcoming paper
\cite{fos-rad}.

\begin{lemma}\label{Lint}
Let $\Omega \subset \mathbb{R}^{d \geq 2}$ be bounded and open, $u \in C^2(\Omega)\cap W^{1,\infty}(\Om)$. Further suppose $\bfal_\delta$ satisfies Assumption \ref{as:alpha}. Then for any $\bfx\in \Om$  and all $\delta > 0$ such that $B_\delta(\bfx) \subset \Omega$,
\begin{equation*}
\begin{split}
&\cL_{\bfal_\delta} [u](\bfx) = \sigma \int_0^1 \int_{B_{\delta}(0)} s \bigg(\Delta u(\bfx+s\mathbf{z})-\Delta u(\bfx)\bigg) \pi(|\mathbf{z}|) d\mathbf{z} ds +\Delta u(\bfx)\,.
\end{split}
\end{equation*}
\end{lemma}

\begin{proof}
We interpret $\cL_{\bfal_\del}$ as the weak nonlocal Laplacian, but since  $u$ is a fortiori in $W^{1,p}(\Om)$ then  $\cL_{\bfal}[u]\in \ltom$ by Proposition \ref{lapexists} and we may use the pointwise original definition according to Proposition \ref{prop:weak-strong}. Because the support of $\mu_{\delta}$ is contained in $B_{\delta}(\bfx)$, we have 
\begin{equation*}
\begin{split}
\cL_{\bfal_\delta} [u](\bfx) &= \sigma \int_{B_\delta(\bfx)} [u(\bfy)-u(\bfx)] \mu_\delta (\bfx,\bfy) d\bfy \\
&= \sigma \int_{B_\delta(\bfx)} \int_0^1 \frac{d}{ds} \left[ u(\bfx+s(\bfy-\bfx))\right]  \mu_\delta (\bfx, \bfy) ds d\bfy \\
&= \sigma \int_{B_\delta(\bfx)} \int_0^1 \nabla u(\bfx+s(\bfy-\bfx)) \cdot (\bfy-\bfx) \mu_\delta (\bfx, \bfy) ds d\bfy .
\end{split}
\end{equation*}
Because $u$ is Lipschitz and $(\bfy-\bfx)\mu_{\del}(\bfx,\bfy)$ is integrable, then we can change the order of integration and then apply substitution $\mathbf{z} = \bfy-\bfx$ to obtain
\begin{equation*}
\cL_{\bfal_\delta} [u](\bfx) = \sigma \int_0^1 \int_{B_\delta (0)} \nabla u(\bfx+s\mathbf{z}) \cdot [\mathbf{z} \mu_{\delta}(|\mathbf{z}|)] d\mathbf{z} ds .
\end{equation*}
With $\pi$ given by \eqref{pi} we know
\begin{equation*}
\nabla_\mathbf{z} \pi (|\mathbf{z}|) =\pi'(|\mathbf{z}|) \frac{\mathbf{z}}{|\mathbf{z}|} = - \mu_\delta(|\mathbf{z}|)\bfz
\end{equation*}
and consequently,
\begin{equation*}
\cL_{\bfal_\delta} [u](\bfx) = - \sigma \int_0^1 \int_{B_\delta (0)} \nabla u(\bfx+s\mathbf{z}) \cdot \nabla_\mathbf{z} \pi(|\mathbf{z}|) d\mathbf{z} ds .
\end{equation*}
Since $\pi(\delta) = 0$, then integration by parts gives
\begin{equation*}
\begin{split}
\cL_{\bfal_\delta} [u](\bfx) =&
\sigma \int_0^1 \int_{B_\delta(0)} \text{div}_\mathbf{z} [\nabla u(\bfx+s\mathbf{z})]  \pi(|\mathbf{z}|) d\mathbf{z} ds\\
=& \sigma \int_0^1 \int_{B_\delta(0)} \Delta u(\bfx+s\mathbf{z})\, s \, \pi(|\mathbf{z}|) d\mathbf{z} ds\,.
\end{split}
\end{equation*}
Let $\om_{d-1}$ be the surface measure of the unit sphere in $\bbR^d$. Using the identity
\begin{equation*}
\begin{split}
 \int_{B_\delta(0)}   \pi(|\mathbf{z}|) d\mathbf{z}  =   \int_0^\delta \pi(r) \om_{d-1} r^{d-1}  d r 
\end{split}
\end{equation*}
we may rewrite $\cL_{\bfal_{\del_n}} [u](\bfx)$ in the desired form:
\begin{equation*}
\begin{split}
\cL_{\bfal_\delta} [u](\bfx) =&  \sigma \int_0^1 \int_{B_\delta(0)} \left(\Delta u(\bfx+s\mathbf{z}) - \Delta u(\bfx) \right) s \pi(|\mathbf{z}|) d\mathbf{z} ds \\
& + \sigma \Delta u(\bfx) \frac{\om_{d-1}}{2} \int_0^\delta  \pi(r)  r^{d-1} dr. 
\end{split}
\end{equation*}
Then apply the fact that $\sig=2d$ and use the definition  \eqref{Cdelta} along with Proposition  \ref{prop:Cdelta} to establish that all constants in the right-most term cancel.
\end{proof}

\begin{theorem}[Convergence of the nonlocal Laplacian]\label{thm:convlap} Suppose $\Omega \subset \mathbb{R}^{d \geq 2}$ is bounded and open, and $\bfal_\delta$ satisfies Assumption \ref{as:alpha}.  Let $u \in C^4(\Omega)\cap W^{4,\infty}(\Om)$ and $\sigma = 2d$ in \eqref{scaledop}. Then there is  $K(u,d)>0$ dependent only on $u$ and $d$ such that
\begin{equation}\label{conv:ball}
|\cL_{\bfal_\delta}[u](\bfx) - \Delta u(\bfx) | \leq K(u,d)\delta^2 \quad \text{ whenever } B_{\del}(\bfx)\subset \Om\,.
\end{equation}
In addition
\begin{equation}\label{eqn:smallunibdd}
\sup_{\bfx\in \Om_{\del}} | \mathcal{L}_{\bfal_\delta}[u](\bfx) | \leq  \| \Delta u\|_{L^\infty(\Omega)} + K(u,d) \delta^2\,. 
\end{equation}
so as $\delta \to 0$
\begin{equation}\label{conv:strong-lap}
\chi_{\Om_\del}\Ln[u] \to \Del u \quad\text{strongly in}\quad L^2(\Om)\,.
\end{equation}

And, if $\supp u \subset \subset \Omega$ with $\delta < \frac{1}{2} \cdot \dist(\supp u, \partial \Omega)$ then
\begin{equation}\label{eqn:compactsuppbdd}
\sup_{\bfx \in \Omega} | \cL_{\bfal_\delta}[u] (\bfx) | \leq \| \Delta u \|_{L^\infty(\Omega)} + K(u,d) \cdot \dist(\supp u, \partial \Omega)^2\,.
\end{equation}
whence as $\delta \to 0$
\begin{equation}\label{conv:lap:compact}
\Ln[u] \to \Del u \quad\text{strongly in}\quad \ltom\,.
\end{equation}

\end{theorem}
\begin{proof}
Let $\delta>0$ be sufficiently small so that $B_\delta(\bfx) \subset \Omega$. Then by Lemma \ref{Lint} 
\begin{equation*}
\begin{split}
\cL_{\bfal_\delta} [u](\bfx) 
&= \sigma \int_0^1 \int_{B_{\delta}(0)} s[\Delta u(\bfx+s\mathbf{z})-\Delta u(\bfx)] \pi(|\mathbf{z}|) d\mathbf{z} ds + \Delta u(\bfx)\,.
\end{split}
\end{equation*}
Thus,
\begin{equation*}
\cL_{\bfal_\delta} [u](\bfx) - \Delta u(\bfx) =  \sigma \int_0^1 \int_{B_{\delta}(0)} s \big[\Delta u(\bfx+s\mathbf{z})-\Delta u(\bfx) \big] \pi(|\mathbf{z}|) d\mathbf{z} ds\,.
\end{equation*}
We will let $P_i(s)$ be a polynomial of degree $i$ in $s$, that will be chosen appropriately.  Rewrite $s = \frac{d}{ds}(\frac{s^2-1}{2})= : P_2'(s)$ and integrate by parts in $s$ using $\frac{d}{ds}(\Delta u(\bfx+s\mathbf{z})) = \Delta\nabla u(\bfx+s\mathbf{z})\cdot \mathbf{z})$ in order to obtain
\begin{equation*}
\begin{split}
\cL_{\bfal_\delta} [u](\bfx) - \Delta u(\bfx) &=  -\sigma \int_{B_{\delta}(0)} \int_0^1 P_2(s)[\Delta \nabla u(\bfx+s\mathbf{z}) ] \cdot \mathbf{z} \pi(|\mathbf{z}|) ds d\mathbf{z} \\
&= -\sigma \int_{B_{\delta}(0)} \int_0^1 P_2(s)[\Delta  \nabla u(\bfx+s\mathbf{z})- \Delta \nabla u(\bfx) ] \cdot \mathbf{z} \pi(|\mathbf{z}|) ds d\mathbf{z}
\end{split}
\end{equation*}
where the last step follows from 
\begin{equation*}
\int_{B_\delta(0)} \mathbf{z} \pi(|\mathbf{z}|) d\mathbf{z} = 0.
\end{equation*}
Since fourth-order derivatives of $u$ are bounded, then we know that $|\Delta \nabla u(\bfx+s\mathbf{z}) - \Delta \nabla u(\bfx)| \leq M (u) |s\mathbf{z}|$ for some constant $M(u)$. Thus, we obtain
\begin{equation*}
\begin{split}
|\cL_{\bfal_\delta} [u](\bfx) - \Delta u(\bfx)| &\leq  \sigma M(u) \int_{B_{\delta}(0)} \int_0^1 |P_2(s) s| |\mathbf{z}|^2 \pi(|\mathbf{z}|) ds d\mathbf{z} \\
&\lesssim \sigma M(u)  \half \omega_{d-1} \int_0^\delta \rho^{n+1} \pi(\rho) d \rho.
\end{split}
\end{equation*}
Finally, use the fact that $\rho^{n+1} \leq \rho^{d-1} \delta^2$ for $\rho\in(0,\delta)$ 
and definition \eqref{Cdelta} along with Proposition \ref{prop:Cdelta} to obtain
\begin{equation}\label{lapbd}
\begin{split}
|\cL_{\bfal_\delta} [u](\bfx) - \Delta u(\bfx)| \leq K(u,d) \delta^2 \sig C = K(u,d) \delta^2\,.
\end{split}
\end{equation}

For any $\bfx \in \Omega_{\delta}$, we know \eqref{lapbd} holds since $B_\delta(x) \subset \Omega$ (note the same $\delta$ is valid all $\bfx \in \Omega_{\delta}$). Thus
\begin{equation*}
\begin{split}
\| \cL_{\bfal_\delta} [u](\bfx) \|_{L^\infty(\Omega_{\delta})} \leq  \| \Delta u\|_{L^\infty(\Omega_{\delta})} + K(u,d) \delta^2\,.
\end{split}
\end{equation*}
This supplies a uniform bound on $\chi_{\Om_\del} \Ln[u]$, so the convergence result \eqref{conv:strong-lap} follows from \eqref{conv:ball} which applies on  $\Om_\del$ and the fact that the measure of $\Om\setminus \Om_\del$ tends to $0$.


To verify \eqref{eqn:compactsuppbdd} take $\delta \leq \frac{1}{2} \cdot \dist(\supp u, \partial \Omega)$. When $\bfx \in \Omega \setminus \Omega_\delta$ notice that $\Delta u(\bfx) = \cL_{\bfal_\delta}[u](\bfx) = 0$. Wheres for all $\bfx \in \Omega_\delta$  estimate \eqref{lapbd} holds (note $\delta$ does not depend on $\bfx$ here). Thus for all $\delta \leq \frac{1}{2} \cdot d(\supp u, \partial \Omega)$ we have 
\[
\sup_{\bfx \in \Omega} | \cL_{\bfal_\delta}[u] (\bfx) | \leq \| \Delta u \|_{L^\infty(\Omega)} + K(u,d) \cdot d(\supp u, \partial \Omega)^2\,.
\]
The assertion of a compact support simplifies 
\eqref{conv:strong-lap} to \eqref{conv:lap:compact}.
\end{proof}

\begin{theorem}[Convergence of the  nonlocal biharmonic]\label{opconvbi} Suppose $\Omega \subset \mathbb{R}^{d \geq 2}$ is bounded and open, $\bfx \in \Omega$, and $\bfal_\delta$ satisfies Assumption \ref{as:alpha}.  Let $u \in C^5(\Omega)\cap W^{5,\infty}(\Om)$ and $\sigma = 2d$. Then there is a constant $K(u,d)>0$  dependent only on $u$ and $d$, such that
\[
|\mathcal{B}_{\bfal_\delta}[u](\bfx) - \Delta^2 u(\bfx) | \leq K(u,d)\delta
\] 
whenever $B_\delta(\bfx) \subset \Omega$.
Moreover, if $\supp u \subset \subset \Omega$ and $\delta < \frac{1}{3} \cdot d(\supp u, \partial \Omega)$ then the estimate is uniform in $\bfx$ and 
\[
\sup_{x \in \Omega} | \mathcal{B}_{\bfal_\delta}[u] | \leq \| \Delta^2 u \|_{L^\infty(\Omega)} + K(u,d) \cdot d(\supp u, \partial \Omega)\,.
\]
\end{theorem}

\begin{proof}
Function $u$ has sufficient regularity to expand the  weak biharmonic according to its definition:
\begin{equation*}
B_{\bfal_\delta} [u] (\bfx) =  \sigma \int_{B_\delta(\bfx)} \left[ \cL_{\bfal_\delta} [u](\bfy) - \cL_{\bfal_\delta} [u](\bfx) \right] \mu_\delta (\bfx,\bfy) d\bfy .
\end{equation*}
Appealing to Lemma \ref{Lint}, and canceling $\sig$ with $C$, yields
\begin{equation}\label{ConvBi}
\begin{split}
B_{\bfal_\delta} [u] (\bfx)
%
=& \sigma \int_{B_\delta(\bfx)} \Bigg[ \Delta u(\bfy)  - \Delta u(\bfx)   \\
& \qquad + \sigma \int_0^1 \int_{B_\delta(0)} s[\Delta u(\bfy+s\mathbf{z}) -\Delta u(\bfy) ] \pi(|\mathbf{z}|) d\mathbf{z} ds  \\ 
& \qquad -  \sigma \int_0^1 \int_{B_\delta(0)} s[\Delta u(\bfx+s\mathbf{z})-\Delta u(\bfx)] \pi(|\mathbf{z}|) d\mathbf{z} ds  \Bigg] \mu_\delta (\bfx,\bfy) d\bfy .
\end{split}
\end{equation}
The first term in the above equation can be simplified using the definition of (scaled) nonlocal Laplacian $\cL_{\bfal_\delta}$:
\begin{equation*}
\sigma \int_{B_\delta(\bfx)} \left[   \Delta u(\bfy)  - \Delta u(\bfx)  \right] \mu_\delta (\bfx,\bfy) d\bfy =  \cL_{\bfal_\delta} \left[\Delta u \right] (\bfx) .
\end{equation*}
From Lemma \ref{Lint}, again, we obtain
\begin{equation*}
\begin{split}
 \cL_{\bfal_\delta} \left[\Delta u\right] (\bfx) =&  \fmbox{ \sigma \int_0^1 \int_{B_\delta(0)} s[\Delta^2 u(\bfx+s\mathbf{z})-\Delta^2 u(\bfx)] \pi(|\mathbf{z}|) d\mathbf{z} ds } +  \Delta^2 u(\bfx).
\end{split}
\end{equation*}
Substituting this back into \eqref{ConvBi} results in
\begin{equation}\label{bibdd2}
\begin{split}
&\mathcal{B}_{\bfal_\delta} [u](\bfx) - \Delta^2 u(\bfx)\\
=&  \sigma \int_{B_\delta(\bfx)}
\bigg[ \sigma \int_0^1 \int_{B_\delta(0)} s[\Delta u(\bfy+s\mathbf{z})-\Delta u(\bfy)] \pi(|\mathbf{z}|) d\mathbf{z} ds   \\
&\qquad \qquad -  \sigma \int_0^1 \int_{B_\delta(0)} s[\Delta u(\bfx+s\mathbf{z})-\Delta u(\bfx)] \pi(|\mathbf{z}|) d\mathbf{z} ds  \bigg] \mu_\delta (\bfx,\bfy) d\bfy \\
&+ \fmbox{\sigma \int_0^1 \int_{B_\delta(0)} s[\Delta^2 u(\bfx+s\mathbf{z})-\Delta^2 u(\bfx)] \pi(|\mathbf{z}|) d\mathbf{z} ds}.  
\end{split}
\end{equation}
Demonstrating that the boxed term is of order $\del$ is a simplified version of the argument necessary for the first two integrals (which incorporate the non-integrable kernel $\mu_\del$). The rest of the proof will focus on the first two summands in \eqref{bibdd2}.

For $s$ introduce 
\[
 F_{s}(\bfx) \dfn \sigma \int_{B_{\delta}(0)}[\Del u (\bfx + s \bfz) - \Del u(\bfx)]\pi(\bfz) d\bfz\,.
\]
and rewrite the first two terms on the right in \eqref{bibdd2} as 
$\ds \int_0^1 s \cL_\del [F_{s}](\bfx)  \; ds  $.
By Lemma \ref{Lint} 
\begin{equation}\label{bibdd3}
\begin{split}
&\int_0^1 s \cL_\del [F_{s}](\bfx)  \;  ds    \\
=& \int_0^1 s \sigma \int_0^1 \int_{B_\delta(0)} \tilde{s} \left( \Delta F_s (\bfx + \tilde{s} \tilde{\bfz}) - \Delta F_s(\bfx)\right) \pi(|\tilde{\bfz}|) d \tilde{\bfz} d \tilde{s} + \Delta F_s(\bfx) ds \\
\end{split}
\end{equation}
Recall $u \in C^5(\Omega)\cap W^{5,\infty}(\Om)$,  whence there exists $M>0$ such that 
\[
\sup_{0 \leq i \leq 4} \big\|D^i \left( u(\bfx+s\mathbf{z}) -  u(\bfx) \right) \big\|_{\bbR^{n^i}} \leq M | s \mathbf{z}|.
\]
Therefore, 
\begin{equation}\label{fbound}
\begin{split}
| \Delta F_s(\bfx) | \leq& \sigma \int_{B_\delta(0)} \left| \Delta^2 u(\bfx + s \bfz) - \Delta^2 u (\bfx) \right|  \pi(\bfz) d \bfz \\
\leq& \sigma M \int_{B_\delta(0)} |s \bfz| \pi(\bfz) d \bfz \\
\leq& \sigma M |s| \omega_{d-1} \int_{B_\delta(0)} r^d \pi(r) dr \quad \quad \text{(change of variables)} \\
\leq& |s| \delta M \sigma \omega_{d-1} \int_{0}^\delta r^{d-1} \pi(r) dr \\  
\leq& 2 |s| \delta M \quad \quad (\text{By definition of } \sigma)\,.
\end{split}
\end{equation}

By applying \eqref{fbound} to \eqref{bibdd3} we obtain the bound
\begin{equation*}\label{finaleq}
\begin{split}
\int_0^1 s \cL_\del [F_{s}](\bfx)  \;  ds \leq& \int_0^1 s \sigma \int_0^1 \int_{B_\delta(0)} \tilde{s} \left( 4 |s| \delta M \right) \pi(|\tilde{\bfz}|) \; d \tilde{z} d \tilde{s} + 2|s| \delta M \; ds \\
\leq& M' \delta \left( \sigma \int_{B_\delta(0)} \pi(|\tilde{\bfz}|) \; d \tilde{\bfz} +1 \right) \\
\leq& 3M' \delta\,.
\end{split}
\end{equation*}

%
%
%
Combining \eqref{bibdd2} with \eqref{finaleq}, we finally arrive at
\begin{equation}\label{bibd}
|\mathcal{B}_{\bfal_\delta} [u](\bfx) - \Delta^2 u(\bfx)| \leq  K(n,u) \cdot \del\,.
\end{equation}

Now to verify the uniform bound. Suppose $\delta < \frac{1}{3} \cdot d(\supp u, \partial \Omega)$. First consider $\bfx \in \Omega \backslash \Omega_\delta$. Since $B_\delta(\bfx) \cap \supp(u) = \emptyset$ we have
\[
\mathcal{L}_{\bfal_\delta}[u](\bfx) =  \sigma \int_{B_\delta(\bfx)} (u(\bfy)-u(\bfx)) \mu(\bfx,\bfy) d\bfy = 0\,.
\]  
Likewise, for any $\bfy \in B_\delta(\bfx)$ we have $B_\delta(\bfy) \cap \supp(u) = \emptyset$. Consequently,
\[
\mathcal{L}_{\bfal_\delta}[u](\bfy) =  \sigma \int_{B_\delta(\bfy)} (u(\bfz)-u(\bfy)) \mu(\bfy,\bfz) d \bfz = 0\,.
\]
We conclude, for any $\bfx \in \Omega \backslash \Omega_\delta$ we have
\[
\mathcal{B}_{\bfal_\delta}[u](\bfx) =  \sigma \int_{B_\delta(\bfx)} (\mathcal{L}_{\bfal_\delta}[u](\bfy) - \mathcal{L}_{\bfal_\delta}[u](\bfx)) \mu(\bfx,\bfy) d \bfy = 0\,.
\]
Thus $\mathcal{B}_{\bfal_\delta}[u]$ and $\Delta^2 u$ agree on $\Omega \backslash \Omega_\delta$. Also notice that $\forall \bfx \in \Omega_\delta$ we have $B_\delta(\bfx) \subset \Omega$ and so \eqref{bibd} holds (note $\delta$ does not depend on $\bfx$ here). Thus for all $\delta < \frac{1}{2} \cdot d(\supp u, \partial \Omega)$ we have 
\[
\sup_{\bfx \in \Omega} |\mathcal{B}_{\bfal_\delta}[u]| \leq \| \Delta^2 u \|_{L^\infty(\Omega)} + K(u,d) \cdot d(\supp u, \partial \Omega).
\]

\end{proof}

\section{Well-posedness of the nonlocal steady state problem}

Throughout this section $\Om$ is any bounded open set in $\bbR^d$ for $d \geq 2$. The first problem we will look at is the nonlocal elliptic biharmonic equation 
\begin{equation}\label{hinged}
\mathcal{B}_{\bfal} [u] = f \txtin \Omega' \\
\end{equation}
with nonlocal equivalent of hinged or clamped boundary conditions  (Definition \ref{def:hinged-clamped}):
\begin{equation}\label{bc}
u \in  \bfH  = \scrH^2_{\bfal, H} \txtor \scrH^2_{\bfal, C}\,.
\end{equation}
Note that both the spaces are topologized by the norm in $\scrH^2_\bfal$.

\begin{proposition}\label{prop:lax-milgram}
Suppose $f\in L^2(\Om')$, and $\bfal$ satisfies Assumption \ref{as:alpha}. There exists a unique (weak) solution $u \in \bfH$ of the  nonlocal PDE \eqref{hinged}.
\end{proposition}
\begin{proof}
We prove this result by using the Lax-Milgram Lemma. For $v \in\bfH$ the associated weak formulation is
\begin{equation*}
\int_{\Omega} \mathcal{B}_{\bfal} [u] v \; d\bfx = \int_{\Omega} fv \; d\bfx .
\end{equation*}

Using the fact that $v = 0$ on $\Om\setminus \Om'$  (regardless of the definition of $\bfH$) and through repeated application of Proposition \ref{parts} we obtain
\[
a[u,v] = \intOm f v\,d\bfx,
\]
where
\begin{equation*}
a[u,v] = \int_{\Omega} \cL_{\bfal} [u] \cL_{\bfal} [v] d\bfx .
\end{equation*}
This bilinear form is continuous since
\begin{equation*}
|a[u,v]|  \leq \| \cL_{\bfal} [u] \|_{L^2(\Omega)} \| \cL_{\bfal} [v] \|_{L^2(\Omega)} \leq \|u\|_{\mathscr{H}^2(\Omega)} \| v\|_{\mathscr{H}^2(\Omega)} .
\end{equation*}
It remains to show coercivity.  Since all functions in $\bfH$ vanish on a fixed collar $\Om\setminus \Om'$, then by Corollary \ref{coro:poincare} there is $C>0$ such that
\begin{equation}\label{LMinq1}
\| u\|_{{L^2}(\Omega)}^2\leq C \| \cG_{\bfal} [u] \|^2_{L^2(\Omega\times \Om)}\,.
\end{equation}
Integration by parts (Proposition \ref{parts}) and H\"{o}lder's inequality yields
\begin{equation}\label{LMinq2}
\| \cG_{\bfal} [u] \|_{L^2(\Omega\times \Om)}^2
 = -\intOm u  \cL_{\bfal} [u]  d\bfx 
\leq \| \cL_{\bfal}[u] \|_{L^2(\Omega)} \|u\|_{L^2(\Omega)}\,.
\end{equation}
By combining \eqref{LMinq1} and \eqref{LMinq2} we conclude that  $\|u\|_{\ltom} \leq C_1 \| \cL[u] \|_{L^2(\Omega)}$ and $\| \cG_{\bfal} [u] \|_{L^2(\Omega\times \Om)} \leq C \| \cL[u] \|_{L^2(\Omega)}$. Thus, there is $c_2>0$ such that for all $u\in \bfH$
\begin{equation*}
a(u,u) = \| \cL[u] \|_{L^2(\Omega)}^2 \geq c_2 \|u\|_{\mathscr{H}^2(\Omega)}^2
\end{equation*} 

Since $\bfH$ is a Hilbert space, by the Lax-Milgram theorem there exists a unique element $u\in \bfH$ satisfying
\begin{equation*}
B[u] = f.
\end{equation*}
\end{proof}


\section{Convergence results}

In this section we will show that the solution to a nonlocal system approximates the corresponding solution of the classical elliptic boundary value problem as the horizon approaches $0$.  Such a study  has been previously carried out for the Navier system in \cite{DuMengesha}. 

Before proceeding to the biharmonic operator we collect several results for the classical and nonlocal Poisson problems:
\begin{equation}\label{ClassicLap}
\left\{
\begin{array}{ll}
\Delta v = f, & \bfx \in \Omega \\
v = 0, & \bfx \in \partial \Omega
\end{array}
\right.
\end{equation}
and its nonlocal analogue:
\begin{equation}\label{NonlocLap}
\left\{
\begin{array}{ll}
\cL_{\bfal_{\delta_n}} [v_n] = f, & \bfx \in \Omega_{\tl{\delta}_n} \\
v_n = 0, & \bfx \in \Omega \setminus \Omega_{\tl{\delta}_n},\quad \tl{\del}_n\geq\del_n\,.
\end{array}
\right.
\end{equation}
Note that for convenience we allow the zero Dirichlet data to be prescribed over a possibly thicker collar than the horizon of the Laplace operator. Recall that we are using the scaled (by $\sig = 2d$) version of the Laplacian \eqref{scaledop}.  
The following result is an adjusted version of \cite[Thm 5.4]{DuMengesha}. 

\begin{theorem}[{cf. \cite[Thm 5.4]{DuMengesha}}]\label{L2} Let $\Om\subset \bbR^d$, $d \geq 2$, be an open bounded set of class $C^1$. Suppose sequence of positive scalars $(\del_n)$  converges to $0$ as $n\to \infty$.
For any $f \in L^2(\Omega)$ the system \eqref{NonlocLap} has a unique solution in $\scrH^1_{\bfal_{\del_n},0}(\Om_{\tl{\del}_n})$, and the sequence of solutions $\left\{ v_n \right\} \subset \mathscr{H}_{\bfal_{\del_n}, 0}^1(\Om_{\tl{\del}_n})$ to \eqref{NonlocLap} converges strongly in $L^2(\Om)$ to $v \in W_0^{1,2}(\Omega)$ where $v$ is the weak (variational) solution of the Laplace equation \eqref{ClassicLap}. In addition, there is $C>0$ so that for all $n $
\begin{equation}\label{bounds-grad-laplace}
\begin{split}
\|v_n \|_{L^2(\Omega)} +  \|\cG_{\bfal_{\delta_n}} [v_n] \|_{L^2(\Omega\times \Om)}
\leq C \|f\|_{L^2(\Omega_{\tl{\del}_n})} \,.
\end{split} 
\end{equation}
Moreover, if $\Om$ is, in addition, either of class $C^2$ or  convex, then the classical elliptic theory (see, for instance, \cite[Thm. IX.25, p. 181]{b:brezis:83} and \cite[pp. 139, 147]{b:grisvard:85}) gives $v \in W^{2,2}(\Om)\cap W^{1,2}_0(\Om)$.
\end{theorem}

\begin{proof}
The well-posedness result for \eqref{NonlocLap}
follows easily via the Lax-Milgram argument as in Proposition \ref{prop:lax-milgram},  using the bilinear form $(u,v)\mapsto \intOm \intOm \cG_\bfal[u]\cdot\cG_\bfal[v] d\bfx d\bfy$ and the nonlocal Poincar\'e estimate (see also \cite{Radu2012} for a different argument).

To invoke Theorem \ref{thm:UniformGrad} let us prove  uniform bounds on $\| \cG_{\bfal_{\delta_n}}[v_n] \|_{L^2(\Om\times \Om)}$ and on $\|v_n\|_{\ltom}$, independent of $n$. Apply Corollary \ref{coro:poincare:0}, nonlocal integration by parts, and H\"{o}lder's inequality:
\begin{equation}\label{a-priori-poisson}
\begin{split}
\|v_n \|_{L^2(\Omega)}^2 \leq& C \|\cG_{\bfal_{\delta_n}} [v_n] \|_{L^2(\Omega\times \Om)}^2 \\
=& - C  \langle \cL_{\bfal_{\delta_n}} [v_n], v_n \rangle_{L^2(\Omega)}  \\
=& -C \langle \cL_{\bfal_{\delta_n}} [v_n], v_n \rangle_{L^2(\Omega_{\tl{\del}_n})}  \\
=& - C \langle f, v_n \rangle_{L^2(\Omega_{\tl{\del}_n})} \\
\leq & C \|f\|_{L^2(\Omega_{\tl{\del}_n})} \|v_n \|_{L^2(\Omega)}\,,
\end{split} 
\end{equation}
where $C$ is independent of $n$. 
Dividing by $\|v_n\|_{\ltom}$ yields a bounds on the latter in terms of $\|f\|_{L^2(\Om_{\tl{\del}_n})}$. That in turn gives a bound on the nonlocal gradient verifying \eqref{bounds-grad-laplace} up to a minor adjustment of the constant.

We conclude by Corollary \ref{coro:UniformGrad} that $\left\{v_n \right\}$ is relatively compact in $L^2(\Omega)$. We will show that every cluster point of $\left\{v_n \right\}$ solves the classical Poisson equation \eqref{ClassicLap} which has a unique solution in $W_0^{1,2}(\Omega)$. Let $(v_n)$ be a convergent subsequence. Consider a test function $\phi \in C_c^\infty(\Omega)$. We may assume that $\tl{\del}_n$ is small enough to ensure that $\supp \phi \subset \Om_{\tl{\del}_n}$. Then via Proposition \ref{parts}
\begin{equation*}
\intOm v_n \cL_{\bfal_{\delta_n}} [\phi] d \bfx = \intOm  \cL_{\bfal_{\delta_n}}   [v_n] \phi d\bfx
=
 \int_{\Om_{\tl{\del}_n}}  \cL_{\bfal_{\delta_n}}   [v_n] \phi d\bfx
 =
 \intOm f \phi d \bfx \,.
\end{equation*}
Since $v_n \to v$ strongly in $\ltom$, and by  the result of Theorem \ref{thm:convlap} applied to the compactly supported function $\phi$  the limit as $n\to \infty$ gives
\begin{equation*}
 \intOm v \Delta \phi \, d\bfx = \intOm f \phi d\bfx
\end{equation*}
as $\delta_n \to 0$, verifying that $v$ is the  distributional solution of the  Poisson problem. Since $v\in W^{1,2}_0(\Omega)$, then via integration by parts we conclude that $v$ is the weak variational solution to the Poisson problem.
\end{proof}

Now  consider a  family of equations
\begin{equation}\label{nonloclap2}
\left\{
\begin{array}{ll}
\cL_{\bfal_{\delta_n}} [v_n] = f_n, & \bfx \in \Omega_{\tl{\delta}_n} \\
v_n = 0, & \bfx \in \Omega \setminus \Omega_{\tl{\delta}_n},\quad \tl{\del}_n \geq \del_n
\end{array}
\right.
\end{equation}
where $v_n \in \mathscr{H}_{\bfal_{\del_n},0}^1(\Om_{\tl{\del}_n})$.

\begin{coro}\label{coro:L2seq} 
The result of  Theorem \ref{L2} holds if in each nonlocal problem \eqref{NonlocLap} and in 
\eqref{bounds-grad-laplace}
we replace $f$ by  $f_n\in \ltom $ assuming that   $f_n\to f$ in $\ltom$.
\end{coro}

\begin{proof}
Since $f_n$ are uniformly bounded in $\ltom$ then as in \eqref{a-priori-poisson} in the proof of Theorem  \ref{L2} we conclude that $\left\{v_n \right\}$ is relatively compact in $L^2(\Om)$. Let $v$ be a cluster point of this sequence. Consider a test function $\phi \in C_c^\infty(\Omega)$. We may assume that $\supp \phi \subset \Om_{\tl{\del}_n}$ for all $n$. By Proposition \ref{parts} we obtain
\begin{equation*}
\intOm v_n \cL_{\bfal_{\delta_n}} [\phi] d \bfx =
\intOm  \cL_{\bfal_{\delta_n}}   [v_n] \phi d \bfx
=
\int_{\Om_{\tl{\del}_n}}  \cL_{\bfal_{\delta_n}}   [v_n] \phi d \bfx
 =
 \intOm f_n \phi d \bfx \,.
\end{equation*}
Since  $f_n \to f$ and $v_n \to v$ in $L^2$, via Theorem \ref{thm:convlap} applied to $\phi$ conclude
\begin{equation*}
 \intOm v \Delta \phi d\bfx = \intOm f \phi  d\bfx
\end{equation*}
From the $W^{1,2}_0(\Om)$ regularity it follows that $v$ is a weak solution of the Poisson problem \eqref{ClassicLap}.
\end{proof}

\subsection{Convergence to classical hinged problem}
We turn to the elliptic problem for the nonlocal biharmonic and the  analog of hinged boundary conditions. Note that second-order boundary condition will be applied to the extended collar. In particular we consider this problem on the space $\scrH^2_{\bfal_{\del_n}}(\Om_{\del}, \Om_{2\del})$ according to the definition \eqref{def:H-hinged}.

\begin{theorem}\label{thm:hinged}
Let $\Om\subset \Rn$, $d \geq 2$, be a bounded domain either of class $C^2$ or convex of class $C^1$.    Suppose sequence of positive scalars $(\del_n)$  converges to $0$ as $n\to \infty$.
The solutions of the nonlocal hinged problems
\begin{equation}\label{nonlocbihinged}
\left\{
\begin{array}{ll}
\mathcal{B}_{\bfal_{\delta_n}} [u_n] = f, & \bfx \in \Omega_{2\delta_n} \\
u_n = 0,  & \bfx\in \Om\setminus \Om_{\del_n} \\
 \cL_{\bfal_{\delta_n}} [u_n] = 0, & \bfx \in \Omega_{\del_n} \setminus \Omega_{2\delta_n}
\end{array}
\right.
\end{equation}
converge  in $L^2(\Om)$ to the weak (variational) solution  $u\in W^{2,2}(\Om)\cap W^{1,2}_0(\Om)$ of 
\begin{equation}\label{locbihinged}
\left\{
\begin{array}{ll}
\Delta^2 u = f, & \bfx \in \Omega \\
u = \Delta u = 0, & \bfx \in \partial \Omega
\end{array}
\right.
\end{equation}
as  $n \to \infty$. Furthermore, if $\Om$ is of class $C^4$ then $u$ is also in $ W^{4,2}(\Om)$.
\end{theorem}
\begin{remark}[Regularity of the domain]
Note that the $C^2$ regularity or convex $C^1$  regularity in the assumption of Theorem \ref{thm:hinged} is only needed to identify the limit of the nonlocal solutions as the classical weak solution of the hinged elliptic problem. 
The nonlocal problem itself is wellposed on non-smooth domains. Moreover, the $L^2$-limit of nonlocal solutions, exists if the boundary is merely $C^1$ (in fact, some relaxation is possible as indicated in Remark 
\ref{rem:regularity}).
\end{remark}


\begin{proof}
\textbf{Step 1.}
Let $\chi_{\del}$ be the characteristic function of $\Om_{\del}$. Set 
\[
v_n(\bfx) \dfn  \chi_{\del}(\bfx)\cL_{\bfal_{\delta_n}} [ u_n](\bfx)\,.
\]
Due to the support of $\mu_{\del_n}(\bfx,\cdot)$ the value of $\cB_{\bfal_{\del_n}}[u_n](\bfx) $  for $\bfx\in \Om_{2\del_n}$ only depends  on  $\cL_{\bfal_{\del_n}}[u_n]$  in $\Om_{\del_n}$. Since in $\Om_{\del_n}$ the functions $v_n$ and $\cL_{\bfal_{\del_n}}[u_n]$ coincide (by definition) then
\[
 \cL_{\bfal_{\del_n}}[v_n] (\bfx) = \cB_{\bfal_{\del_n}} [u_n](\bfx) = f(\bfx) \txtfor \bfx\in \Om_{\del_{2n}}\,.
\]
When $\bfx\in \Om\setminus \Om_{\del_n}$ then $\bfv_n=0$, and if $\bfx \in \Om_{\del_n}\setminus \Om_{2\del_n} $, then $v_n(\bfx) = \cL_{\bfal_{\del_n}}[u_n](\bfx) = 0$ by the second boundary condition in \eqref{nonlocbihinged}. Hence $v_n$ satisfies
\begin{equation*}
\left\{
\begin{array}{ll}
\cL_{\bfal_{\delta_n}} [v_n] = f, & \bfx \in \Omega_{2\delta_n} \\
v_n = 0, & \bfx \in \Omega \setminus \Omega_{2\delta_n}\,.
\end{array}
\right.
\end{equation*}
By Theorem \ref{L2} (with $\tl{\del}_n=2\del_n \geq \del_n$) we know that $v_n \to v$ in $\ltom $ with  $v \in W_0^{1,2}(\Omega) \cap W^{2,2}(\Om)$, where $\Delta v = f$. 

\smallskip

\textbf{Step  2.} By definition of $v_n$, we also know that $u_n$ solves
\[
\left\{
\begin{array}{ll}
\cL_{\bfal_{\delta_n}} [u_n] = v_n, & \bfx \in \Omega_{\delta_n} \\
u_n = 0, & \bfx \in \Omega \setminus \Omega_{\delta_n}\,.
\end{array}
\right.
\]
According to Corollary \ref{coro:L2seq} (now $\tl{\del}_n=\del_n$),  $u_n \to u$ in $L^2(\Omega)$, where $u \in W^{1,2}_0(\Om)$ and $u$ is the weak solution to $\Delta u = v$.  Then it follows that
\[
\Delta^2 u = \Delta v = f
\]
 in the sense of distributions. Since $u \in W^{1,2}_0(\Omega)$ and $\Del u = v \in \ltom$, then  either $C^2$ regularity of the domain or its convexity ensures  $u \in W^{2,2}(\Om)\cap W^{1,2}_0(\Om)$. This fact, along with  $\Del u\in W^{1,2}_0(\Om)$ via integration by parts of the distributional identity $\Del^2 u = f$ verifies that $u$ is, in fact, a weak solution of \eqref{locbihinged}.

Moreover, from the elliptic theory  
follows that if the boundary is of class $C^4$ then   $u\in W^{4,2}(\Om)$   \cite[Sec 3A, pp. 282--284]{b:las-tri:v1} (the latter, in fact, deals with smooth boundary, but it suffices to consider the order of boundary regularity  which matches that of the operator, namely, $C^4$).
\end{proof}

\subsection{Convergence to classical solution in clamped problem}
The next result deals with a nonlocal version of the clamped biharmonic problem. This time the boundary conditions are up to the first order, and are both prescribed on the ``original" collar $\Om\setminus \Om_{\del_n}$. As indicated earlier in the analysis of nonlocal clamped boundary conditions, this condition in fact affects the extended collar too 
 \eqref{def:H-hinged} essentially forcing the solution  to be zero on the extended collar.

\begin{remark}[Distinction between nonlocal hinged and clamped]
Note that $u=0$ on $\Om\setminus \Om_{2\del}$ does not provide  $\cL_{\bfal_{\del_n}}[u_n] =0$ on $\Om_\del\setminus \Om_{2\del}$ as imposed in the hinged case, since the nonlocal Laplacian on $\Om_\del\setminus \Om_{2\del}$ subdomain also draws information from inside $\Om_{2\del}$. 

Had all boundary conditions in \eqref{nonlocbihinged} were instead imposed on the same collar $\Om\setminus \Om_\del$, then those conditions would have been a consequence of the nonlcoal clamped data on $\Om_\del$.
\end{remark}

\begin{equation}\label{nonlocbiclamped}
\left\{
\begin{array}{ll}
\mathcal{B}_{\bfal_{\delta_n}} [u_n] = f, & \bfx \in \Omega_{2\delta} \\
u_n = \mathcal{N}_{\bfal_{\delta_n}} [ \cG_{\bfal_{\delta_n}} [u_n]] = 0, & \bfx \in \Omega \setminus \Omega_{\delta_n}\,, \quad (\text{equiv. }  u_n = 0,\; \bfx \in\Om\setminus\Om_{2\del})
\end{array}
\right.
\end{equation}
and its classical analog:
\begin{equation}\label{locbiclamped}
\left\{
\begin{array}{ll}
\Delta^2 u = f, & \bfx \in \Omega \\
\ds u = \frac{\partial u}{\partial \nu} = 0, & \bfx \in \partial \Omega\,.
\end{array}
\right.
\end{equation}

\begin{theorem}\label{clampedthm} Let $\Om\subset \Rn$, $d \geq 2$, be bounded open of class $C^2$ or convex of class $C^1$.
 Suppose sequence of positive scalars $(\del_n)$  converges to $0$ as $n\to \infty$.
For all $f \in L^2(\Omega)$, the sequence of solutions $\left\{ u_n \right\} \subset \scrH^2_{\bfal_{\del_n}, C}(\Om_{\del_n})$ to \eqref{nonlocbiclamped} converges strongly in $L^2(\Omega)$ to $v \in W_0^{2,2}(\Omega)$, which is a weak solution of \eqref{locbiclamped}, as $n\to \infty$.  Moreover if  $\Om$ is of class $C^4$,  then $u$ is the regular solution to this elliptic problem: $u\in W^{4,2}(\Om)$.
\end{theorem}
\begin{remark}[Regularity of the domain]\label{rem:clamped-regularity}
Again, the regularity of the domain requested in Theorem \ref{clampedthm} is only invoked to identify the limit of the nonlocal solutions as the solution of the corresponding classical elliptic problem.
However, for a fixed horizon $\del>0$, the solution to the nonlocal problem furnished by Proposition \ref{prop:lax-milgram} does not place any restrictions on the regularity of the domain. Furthermore, the $L^2$ limit of the non-local approximations is well-defined on $C^1$ domains (just as in the hinged case;  see also Remark \ref{rem:regularity}).
\end{remark}

\begin{proof}
\textbf{Step 1.}
We plan to invoke Theorem \ref{thm:UniformGrad} (Corollary \ref{coro:UniformGrad}), so we need to demonstrate an upper bound on $\| \cG_{\bfal_{\delta_n}} [u_n]\|_{L^2(\Omega)}$, independent of $n$. 
Apply the Poincar\'e-type inequality of Corollary \ref{coro:poincare:0} and
nonlocal integration by parts  (Proposition \ref{parts}, applied in the weak sense as indicated in Remark \ref{rem:by-parts-weak-strong}):
\begin{equation}\label{coerc}
\begin{split}
c\|u_n\|_{L^2(\Omega)}^2 \leq \| \cG_{\bfal_{\delta_n}}[u_n] \|_{L^2(\Omega\times\Om)}^2 
=& -\intOm \cL_{\bfal_{\delta_n}}[u_n] u_n d\bfx \\
\leq& \| \cL_{\bfal_{\delta_n}}[u_n] \|_{L^2(\Omega)} \| u_n \|_{L^2(\Omega)}
\end{split}
\end{equation}
for some $c>0$. Hence
\[
c\|u_n\|_{L^2(\Omega)} \leq \| \cL_{\bfal_{\delta_n}}[u_n] \|_{L^2(\Omega)} 
\]
\[
 \|  \cG_{\bfal_{\delta_n}}[u_n] \|_{L^2(\Omega\times\Om)}^2 \leq \frac{1}{c}\| \cL_{\bfal_{\delta_n}}[u_n] \|_{L^2(\Omega)}^2\,.
\]  
Thus:
\begin{equation}\label{coerc2}
\begin{split}
\|u_n \|_{L^2(\Omega)}^2 \leq& \frac{1}{c} \|\cG_{\bfal_{\delta_n}} [u_n] \|_{L^2(\Omega\times \Omega)}^2 
\leq \frac{1}{c^2} \|\cL_{\bfal_{\delta_n}} [u_n] \|_{L^2(\Omega)}^2 \\
=& \frac{1}{c^2} \left| \langle \mathcal{B}_{\bfal_{\delta_n}}[u_n], u_n \rangle_{L^2(\Om)} \right|\,.
\end{split}
\end{equation}
Since the clamped space \eqref{H-clamped-delta} enforces zero data on the extended collar $\Om\setminus\Om_{2\del}$ we continue the estimate \eqref{coerc2}
\begin{equation}\label{coerc3}
\begin{split}
\ldots \leq & \frac{1}{c^2}
\left| \langle \cB_{\bfal_{\del_n}}[u_n],u_n\rg_{L^2(\Om_{2\del})} \right|
= \frac{1}{c^2} \left| \langle f, u_n \rangle_{L^2(\Omega)} \right| \\
\leq& \frac{1}{c^2} \|f\|_{L^2(\Omega)} \|u_n \|_{L^2(\Omega)}
\end{split} 
\end{equation}
where $c$ is independent of $n$. Therefore,  $( u_n )$ is  bounded in $L^2(\Omega)$ and so is $( \cG_{\bfal_{\delta_n}} [u_n] )$.  By Corollary \ref{coro:UniformGrad},  $\left\{u_n \right\}$ is relatively compact in $L^2(\Omega)$ and if $u$ is a cluster point of $\left\{u_n \right\}$, then $u \in W^{1,2}_0(\Omega)$. 

It remains to show that any cluster point of $\left\{ u_n \right\}$  is the unique weak solution of \eqref{locbiclamped}.   Pick a test function $\phi \in C_c^\infty(\Omega)$. We may assume that $\del_n$ is small enough so that $\supp \phi \subset \Om_{3\del_n}$ (note the factor of $3$ instead of just $2$). By Proposition \ref{parts} we then have
\begin{equation*}
\langle \mathcal{B}_{\bfal_{\delta_n}} [\phi],u_n \rangle_\ltom = \langle \phi, \mathcal{B}_{\bfal_{\delta_n}} [u_n] \rangle_{\ltom}  = \langle \phi, f \rangle_{\ltom}
\end{equation*}
Use the fact that $u_n \to u$ strongly in $\ltom$ and Theorem \ref{opconvbi}
to conclude that
\begin{equation*}
\intOm \Delta^2 \phi \,u\; d\bfx = \int \phi f d\bfx
\end{equation*}
as $n\to \infty$. This verifies that $u \in W^{1,2}_0(\Om)$  is a distributional solution to
\[
 \Del^2 u = f\,.
\]
\smallskip

\textbf{Step 2.} Let's prove  $u\in W^{2,2}(\Om)$. Set $  z_n = \cL_{\bfal_{\del_n}}[u_n] $. From  \eqref{coerc2}--\eqref{coerc3} 
 \[
 \|  z_n \|^2_{\ltom} 
\leq  \|f\|_{\ltom}^2 \|u_n \|_{\ltom}^2\,.
 \]
Since $(u_n)$ is bounded in $\ltom$ then so is $(z_n)$. Conclude that $(z_n)$ converges weakly to some $z$ in $\ltom$. At the same time,  for $\phi \in C_c^\infty(\Om)$ nonlocal integration by parts gives
\[
\lg z_n, \phi \rg_\ltom = \lg u_n, \Ln[\phi]\rg_\ltom\,.
\]
We know $u_n \to u$ strongly in $\ltom$ and  because $\phi$ is compactly supported, by Theorem \ref{thm:convlap} we also have $\Ln[\phi] \to \Del \phi $ strongly in $\ltom$. And since $z_n$ converges weakly then in the limit $n\to \infty$ this identity becomes
\[
\lg z,  \phi \rg_{\ltom} = \lg u, \Del \phi \rg_\ltom\,.
\]
So $z\in \ltom$ is the distributional Laplacian of $u$. Since $u\in W^{1,2}_0(\Om)$ and the domain $\Om$ is assumed either $C^2$ or convex (on top of $C^1$), then the elliptic regularity ensures that $u\in W^{2,2}(\Om)\cap W^{1,2}_0(\Om)$.

\smallskip

\textbf{Step 3.} To prove that $u$ is the weak variational solution of the classical biharmonic problem it remains to establish that $u$ has zero normal trace. For this purpose  consider a family of test functions in $C^{4}(\Om)\cap W^{4,\infty}(\Om)$ to which we will ultimately  apply Theorem \ref{thm:convlap}. Due to the at least $C^1$ regularity of the domain  (in fact, for this statement even a bounded Lipschitz domain would suffice) the function space $C^\infty(\cl{\Om})$ is dense in $W^{1,2}(\Om)$ \cite[Thm. 3, p. 127]{b:eva-gar:92}. In turn, the traces of $W^{1,2}(\Om)$ functions are onto $W^{1,1/2}(\Gam)$ (e.g., for uniform $C^1$ domains) and, thus, dense in $L^2(\Gam)$. Hence the traces of $C^{4}(\Om)\cap W^{4,\infty}(\Om)$ functions are dense in $L^2(\Gam)$.

Fix  $\psi \in C^{4}(\Om)\cap W^{4,\infty}(\Om)$.
We  showed in Step 2 that $\Del u \in \ltom $, and by Theorem  \ref{thm:convlap} we know that $z_n =\Ln[u_n]$ converges to $\Del u$ weakly in $\ltom$.  Thus 
\begin{equation}\label{L-n-u-v-to-Del-u-v}
\lim_{n\to \infty}\intOm \cL_{\bfal_{\delta_n}} [ u_n] \psi\, d \bfx = \intOm \Delta u \psi\, d \bfx\,.
\end{equation}
The goal is to verify that
\[
\intOm \Delta u \psi\, d \bfx = \intOm u \Delta \psi \, d \bfx\,.
\]
Invoke again  nonlocal integration by parts:
\begin{equation*}
\begin{split}
&\intOm \cL_{\bfal_{\delta_n}}[u_n] (\bfx)\psi(\bfx)d \bfx = \intOm u_n(\bfx) \cL_{\bfal_{\delta_n}}[\psi](\bfx) d \bfx \\
=& \left(\int_{\Omega_{\delta_n}}  + \int_{\Om\setminus \Om_{\del_n}} 
   \right)  u_n(\bfx)\Ln[\psi](\bfx)d\bfx\,.
\end{split}
\end{equation*}
The second of the two integrals vanishes since $u_n(\bfx) = 0$ on the collar $\Om\setminus \Om_\del $. Whereas  $\chi_{\Om_{\del_n}} \Ln[\psi]$ converges strongly to $\Del \psi$ by Theorem \ref{thm:convlap}. Along with the strong convergence $u_n \to u$ in $\ltom$ this argument gives
\[
\lim_{n \to \infty}\intOm \cL_{\bfal_{\delta_n}}[u_n]  \psi (\bfx)d \bfx = \int_{\Omega} u  \Delta \psi d \bfx.
\]
In combination with \eqref{L-n-u-v-to-Del-u-v} we conclude
\[
\int_{\Omega} \Delta u \psi d \bfx = \int_{\Omega} u  \Delta \psi d \bfx\,.
\]
Since $u=0$ on the boundary and the traces of  $ C^4(\Om)\cap W^{4,\infty}(\Om)$ functions are dense in $L^2(\Gam)$, then the above identity with arbitrary $\psi \in C^4(\Om)\cap W^{4,\infty}(\Om) $  implies 
\[
 \Dn{u}  =0\,.
\]
Thus $u$ is a weak solution to the elliptic problem \eqref{locbiclamped}.   When the domain is of class $C^4$ the additional $W^{4,2}(\Om)  $ regularity follows from the same remarks as at the end at the end of the proof of Theorem \ref{thm:hinged}.
\end{proof}

\bibliographystyle{abbrv} 
\bibliography{JetBib1}
\end{document}